\RequirePackage{luatex85}
\documentclass[]{amsart}
\usepackage{fontspec}
\usepackage{mathcmd}
\usepackage{graphicx}
\usepackage{mathptmx}
\usepackage{mathtools}
\usepackage{amssymb}
\usepackage{amsmath}
\usepackage{amsfonts}
\usepackage{latexsym}
\usepackage{amsthm}
\usepackage[utf8]{inputenc}

\usepackage[all]{xy}

\usepackage{appendix}

\usepackage[T1]{fontenc}
\usepackage{lmodern}
\usepackage[english]{babel}

\usepackage{blindtext}

\def \tt#1{\text{\rm #1}}

\def \etale {\'{e}tale }

\def \ab {\tt{ab}}

\def \Aut {\tt{Aut}}

\def \Inn {\tt{Inn}}
\def \Out {\tt{Out}}

\def \Z {\mathbb{Z}}

\def \abt {\text{\rm ab-t}}
\def \sp {\tt{Spec}}
\def \cusp {\text{\rm Cusp}}
\def \id {\tt{Id}}
\def \T {\tt{T}}

\def \root {\tt{Root}}

\def \Cusp {\text{\rm Cusp}}

\def \diag {\tt{diag}}

\def \card {\tt{Card}}

\def \obj {\tt{Obj}}

\def \hom {\tt{Hom}}

\def \sch {\tt{Sch}}

\def \set {\tt{set}}

\def \gal {\tt{Gal}}

\def \sche {\tt{Scheme}}
\def \base {\tt{Base}}
\def \pr {\tt{pr}}
\def \surjto {\twoheadrightarrow}
\def \injto {\hookrightarrow}
\def \cl {\tt{cpt}}
\def \cavc {\mathfrak{CAVC}}
\def \ig {\mathfrak{IG}}
\def \mf {\mathfrak}

\def \simto {\xrightarrow{\sim}}
\def \cov {\tt{cov}}
\def \bc {\tt{b-c}}
\def \uc {\tt{UC}}

\usepackage{cleveref}
\crefname{section}{\S}{\S}
\Crefname{section}{\S}{\S}

\crefname{theo}{Theorem}{Theorem}
\Crefname{theo}{Theorem}{Theorem}
\crefname{defi}{Definition}{Definition}
\Crefname{defi}{Definition}{Definition}
\crefname{prop}{Proposition}{Proposition}
\Crefname{prop}{Proposition}{Proposition}

\crefname{lemm}{Lemma}{Lemma}
\Crefname{lemm}{Lemma}{Lemma}

\crefname{coro}{Corollary}{Corollary}
\Crefname{coro}{Corollary}{Corollary}

\crefname{remm}{Remark}{Remark}
\Crefname{remm}{Remark}{Remark}

\crefname{thma}{Theorem A}{Theorem A}
\Crefname{thma}{Theorem A}{Theorem A}

\crefname{thma0}{Theorem A}{Theorem A}
\Crefname{thma0}{Theorem A}{Theorem A}

\crefname{thma1}{Theorem B}{Theorem B}
\Crefname{thma1}{Theorem B}{Theorem B}

\crefname{thma2}{Theorem C}{Theorem C}
\Crefname{thma2}{Theorem C}{Theorem C}

\theoremstyle{theorem}
\newtheorem{lemm}{Lemma}[section]
\newtheorem{theo}[lemm]{Theorem}
\newtheorem{prop}[lemm]{Proposition}
\newtheorem{coro}[lemm]{Corollary}

\newtheorem*{thma'}{Theorem A$'$}
\newtheorem*{thmb}{Theorem D} 

\newtheorem*{thmc'}{Theorem C$'$}

\newtheorem*{thma0}{Theorem A}
\newtheorem*{thma1}{Theorem B}
\newtheorem*{thma2}{Theorem C}

\theoremstyle{remark}
\newtheorem{remm}[lemm]{Remark}

\theoremstyle{definition}
\newtheorem{defi}[lemm]{Definition}

\usepackage{lipsum}

\newcommand{\longtitle}[1]{
	\ifodd\value{page}
	\protect\parbox{0.97\linewidth}{#1}
	\else
	\fi
}

\calclayout

\title[The Absolute Anabelian Geometry of Virtual Curves of Arbitrary Genus]{The Absolute Anabelian Geometry of Virtual Curves of Arbitrary Genus}
\author{Zeming Sun}

\begin{document}

\maketitle


\section{Introduction}\label{Introductions}

The objective of this paper is to further study the \emph{pointed virtual curves} defined in \cite{SZM}, and to introduce the notion of 
an inclusion of \emph{CAVC-type} (cf. \cref{defi_CAVCs}(ii); here, ``CAVC'' stands for ``cuspidal-admissible virtual curve'') 
and the notion of a \emph{virtual decuspidaloid} (cf. \cref{construction_CAVC}). 
Roughly speaking, the notion of an inclusion of CAVC-type is a group-theoretic abstraction of the inclusion of the geometric virtual fundamental group $\Delta_{\mathcal{C}}$ of a pointed virtual curve $\mathcal{C}$ into the virtual fundamental group $\Pi_{\mathcal{C}}$ (cf. \cref{Notations and Terminologies} the discussion entitled ``Virtual Curves''), while 
the notion of a virtual decuspidaloid is a categorical-theoretic abstraction of the decuspidalization operation on virtual 
fundamental groups of pointed virtual curves developed in [\cite{SZM}, \S5]. 

Our initial theorem extends [\cite{SZM}, Theorem A] from the genus-zero situation to curves of arbitrary genus. Whereas [\cite{SZM}, Theorem A] concerns hyperbolic curves of genus $0$ under suitable hypotheses, we now distinguish the following three mutually exclusive cases, where $g$ denotes the genus and $r$ denotes the number of cusps (cf. \cref{Notations and Terminologies}, the discussion entitled ``Fields, Schemes and Curves''):
\begin{description}
    \item[Case 0] $g=0$ and $r\geq 3$ (cf. Theorem A);
    \item[Case 1] $g=1$ and $r\geq 1$ (cf. Theorem B);
    \item[Case 2] $g\geq 2$ (cf. Theorem C).
\end{description}

\begin{thma0}\label{a0}(cf. 
[\cite{SZM}, Theorem 6.15(i)(iv)(v)]; \cref{reconstruction_of_section_special_sym}(iii).)

    Let $k$ be either a number field or a mixed-characteristic local field; $\bar{k}$ an algebraic closure of $k$; $G_{k}$ the Galois group $\gal(\bar{k}/k)$; $X$ a hyperbolic curve over $k$ of type $(0,r)$, where $r\geq 3$; $s:G_{k}\rightarrow \Pi_{X}$ a section of the natural surjection $\Pi_{X}\surjto G_{k}$. Write $\Pi$ for the virtual fundamental group $\Pi_{[\pr^{2/1}_{X},s]}$; $\Delta$ for the geometric virtual fundamental group $\Delta_{[\pr^{2/1}_{X},s]}$
    (cf. \cref{Notations and Terminologies}, the discussion entitled ``Configuration Spaces''; \cref{Notations and Terminologies}, the discussion entitled ``Virtual Curves'').
    Then the following assertions hold:

        (i) One may construct a field 
        \begin{displaymath}
            \base(\Pi)	
	\end{displaymath}
	from the abstract profinite group $\Pi$ functorially with respect to isomorphisms of topological groups such that the action of $\Pi$ on $\base(\Pi)$ functorially induced by the conjugation action of $\Pi$ on itself factors through the quotient $\Pi\surjto G_{k}$. Moreover, there exists a $G_{k}$-equivariant isomorphism $\base(\Pi)\simto \bar{k}$.

	(ii) The isomorphism class of the function field $K(X)$ of $X$ is determined by the isomorphism class of the abstract profinite group $\Pi$.
	
	(iii) Suppose that $X\simto\mathbb{P}^{1}_{k}\setminus\{0,1,\infty\}$. Recall that the group of automorphisms $\Aut(\Pi_{X})$ of the profinite group $\Pi_{X}$ acts naturally on the set of sections of the surjection $\Pi_{X}\surjto G_{k}$, and that there exists a natural isomorphism $\Aut(X)\simto \Out(\Pi_{X})$ (cf. [\cite{MZK3}, Theorem 2.6(v)(vi)]; [\cite{MZK2}, Theorem 1.9]). Then the $\Aut(\Pi_{X})$-orbit of the section $s$ is determined by the isomorphism class of the abstract profinite group $\Pi$. Moreover, the isomorphism class of the scheme $X$ is determined by the isomorphism class of the abstract profinite group $\Pi$.

\end{thma0}

\begin{thma1}\label{a1} (cf. \cref{reconstruction_of_section_special_sym}(iv); \cref{reconstruction_algorithm_genus_1}; \cref{validity_of_descent_genus_1}.)

    Let $k$ be either a number field or a mixed-characteristic local field; $\bar{k}$ an algebraic closure of $k$; $G_{k}$ the Galois group $\gal(\bar{k}/k)$; $X$ a hyperbolic curve over $k$ of type $(1,r)$, where $r\geq 1$; $s:G_{k}\rightarrow \Pi_{X}$ a section of the natural surjection $\Pi_{X}\surjto G_{k}$. Write $\Pi$ for the virtual fundamental group $\Pi_{[\pr^{2/1}_{X},s]}$; $\Delta$ for the geometric virtual fundamental group $\Delta_{[\pr^{2/1}_{X},s]}$; $X^{\cl}$ for the unique smooth compactification of $X$ (cf. \cref{Notations and Terminologies}, the discussion entitled ``Configuration Spaces''; \cref{Notations and Terminologies}, the discussion entitled ``Virtual Curves''). Assume further that $X$ admits a $k$-rational point. Then the following assertions hold:

        (i) One may construct a field 
        \begin{displaymath}
            \base(\Pi)	
	\end{displaymath}
	from the abstract profinite group $\Pi$ functorially with respect to isomorphisms of topological groups such that the action of $\Pi$ on $\base(\Pi)$ functorially induced by the conjugation action of $\Pi$ on itself factors through the quotient $\Pi\surjto G_{k}$. Moreover, there exists a $G_{k}$-equivariant isomorphism $\base(\Pi)\simto \bar{k}$.

        (ii) One may construct a hyperbolic curve
        \begin{displaymath}
            f: \sche(\Pi) \to \sp(\base(\Pi))
        \end{displaymath}
        over $\base(\Pi)$
        from the abstract profinite group $\Pi$ functorially with respect to isomorphisms of topological groups such that the action of $\Pi$ on $\sche(\Pi)$ functorially induced by the conjugation action of $\Pi$ on itself factors through the quotient $\Pi\surjto G_{k}$.
        Denote by $(\sche(\Pi))^{\cl}$ the unique smooth compactification of $\sche(\Pi)$.
        There exists a $G_{k}$-equivariant isomorphism $(\sche(\Pi))^{\cl}\simto X^{\cl}\times_{k}\bar{k}$ compatible with the isomorphism $\base(\Pi)\simto \bar{k}$ of (i). In particular, the isomorphism class of the function field $K(X)$ of $X$ is determined by the isomorphism class of the abstract profinite group $\Pi$.
	
	(iii) Suppose that $r=1$. Recall that the group of automorphisms $\Aut(\Pi_{X})$ of the profinite group $\Pi_{X}$ acts naturally on the set of sections of the surjection $\Pi_{X}\surjto G_{k}$, and that there exists a natural isomorphism $\Aut(X)\simto \Out(\Pi_{X})$ (cf. [\cite{MZK3}, Theorem 2.6(v)(vi)]; [\cite{MZK2}, Theorem 1.9]; [\cite{RNSPM}, Theorem D]). Then the $\Aut(\Pi_{X})$-orbit of the section $s$ is determined by the isomorphism class of the abstract profinite group $\Pi$. Moreover, the isomorphism class of the scheme $X$ is determined by the isomorphism class of the abstract profinite group $\Pi$.

\end{thma1}

\begin{thma2}\label{a2}(cf. \cref{reconstruction_of_section_special_sym}(v); \cref{reconstruction_algorithm_genus_2}.)

    Let $k$ be either a number field or a mixed-characteristic local field; $\bar{k}$ an algebraic closure of $k$; $G_{k}$ the Galois group $\gal(\bar{k}/k)$; $X$ a hyperbolic curve over $k$ of type $(g,r)$, where $g\geq 2$; $s:G_{k}\rightarrow \Pi_{X}$ a section of the natural surjection $\Pi_{X}\surjto G_{k}$. Write $\Pi$ for the virtual fundamental group $\Pi_{[\pr^{2/1}_{X},s]}$; $\Delta$ for the geometric virtual fundamental group $\Delta_{[\pr^{2/1}_{X},s]}$; $X^{\cl}$ for the unique smooth compactification of $X$ (cf. \cref{Notations and Terminologies}, the discussion entitled ``Configuration Spaces''; \cref{Notations and Terminologies}, the discussion entitled ``Virtual Curves''). Then the following assertions hold:

        (i) One may construct a field 
        \begin{displaymath}
            \base(\Pi)	
	\end{displaymath}
	from the abstract profinite group $\Pi$ functorially with respect to isomorphisms of topological groups such that the action of $\Pi$ on $\base(\Pi)$ functorially induced by the conjugation action of $\Pi$ on itself factors through the quotient $\Pi\surjto G_{k}$. Moreover, there exists a $G_{k}$-equivariant isomorphism $\base(\Pi)\simto \bar{k}$.

        (ii) One may construct a hyperbolic curve
        \begin{displaymath}
            f: \sche(\Pi) \to \sp(\base(\Pi))
        \end{displaymath}
        over $\base(\Pi)$
        from the abstract profinite group $\Pi$ functorially with respect to isomorphisms of topological groups such that the action of $\Pi$ on $\sche(\Pi)$ functorially induced by the conjugation action of $\Pi$ on itself factors through the quotient $\Pi\surjto G_{k}$. Moreover, there exists a $G_{k}$-equivariant isomorphism $\sche(\Pi)\simto X^{\cl}\times_{k}\bar{k}$ compatible with the isomorphism $\base(\Pi)\simto \bar{k}$ of (i). In particular, the isomorphism class of the function field $K(X)$ of $X$ is determined by the isomorphism class of the abstract profinite group $\Pi$.

        (iii) Denote by $\Pi^{\cl}$ the quotient $\Pi/(\cusp(\Pi)\coloneqq \cusp([\pr^{2/1}_{X},s])$ of $\Pi$ (cf. [\cite{SZM}, Definition 4.19(i)]; \cref{Notations and Terminologies}, the discussion entitled ``Groups and Topologies''). Then one may construct a field
        \begin{displaymath}
            \uc(\Pi)
        \end{displaymath}
        (where ``UC'' stands for ``universal covering''),
        together with a field inclusion
        \begin{displaymath}
            h: K(\sche(\Pi)) \injto \uc(\Pi)
        \end{displaymath}
        (where $K(\sche(\Pi))$ denotes the function field of $\sche(\Pi)$), from the abstract profinite group $\Pi$ functorially with respect to isomorphisms of topological groups such that the action of $\Pi$ on $\sche(\Pi)$ functorially induced by the conjugation action of $\Pi$ on itself factors through the quotient $\Pi\surjto \Pi^{\cl}$. Note that $\Pi^{\cl}$ is naturally isomorphic to the \etale fundamental group $\Pi_{X^{\cl}}$ (cf. the argument given in step (b) of \cref{reconstruction_algorithm_genus_2}). Let $\uc$ be an FFUC of $X^{\cl}$ (cf. \cref{Notations and Terminologies}, the discussion entitled ``Fundamental Groups''). Then there exists a $\Pi$-equivariant isomorphism $\uc(\Pi)\simto \uc$ compatible with the isomorphism $\sche(\Pi)\simto X^{\cl}\times_{k}\bar{k}$ of (ii).
	
	(iv) Suppose that $r=0$. Recall that the group of automorphisms $\Aut(\Pi_{X})$ of the profinite group $\Pi_{X}$ acts naturally on the set of sections of the surjection $\Pi_{X}\surjto G_{k}$, and that there exists a natural isomorphism $\Aut(X)\simto \Out(\Pi_{X})$ (cf. [\cite{MZK3}, Theorem 2.6(v)(vi)]; [\cite{MZK2}, Theorem 1.9]; [\cite{RNSPM}, Theorem D]). Then the $\Aut(\Pi_{X})$-orbit of the section $s$ is determined by the isomorphism class of the abstract profinite group $\Pi$. Moreover, the isomorphism class of the scheme $X$ is determined by the isomorphism class of the abstract profinite group $\Pi$.

\end{thma2}

Out next result is a criterion about the ``geometricity'' of certain virtual curves, which is an interesting application of [\cite{RNSPM}, Theorem F].

\begin{thmb}\label{14to04}(cf. \cref{real_thmb}.)

    Let $k$ (respectively, $k'$) be either a number field or a mixed-characteristic local field; $\bar{k}$ (respectively, $\bar{k}'$) an algebraic closure of $k$ (respectively, $k'$); $G_{k}$ (respectively, $G_{k'}$) the Galois group $\gal(\bar{k}/k)$ (respectively, $\gal(\bar{k}'/k')$); $X$ (respectively, $X'$) a hyperbolic curve over $k$ (respectively, $k'$) of type $(0,3)$ (respectively, type $(1,3)$); $s:G_{k}\rightarrow \Pi_{X}$ (respectively, $s':G_{k'}\rightarrow \Pi_{X'}$) a section of the natural surjection $\Pi_{X}\surjto G_{k}$ (respectively, $\Pi_{X'}\surjto G_{k'}$). Write $\Pi$ (respectively, $\Pi'$) for the virtual fundamental group $\Pi_{[\pr^{2/1}_{X},s]}$ (respectively, $\Pi_{[\pr^{2/1}_{X'},s']}$); $\Delta$ (respectively, $\Delta'$) for the geometric virtual fundamental group $\Delta_{[\pr^{2/1}_{X},s]}$ (respectively, $\Delta_{[\pr^{2/1}_{X'},s']}$). Assume that there exists a open group injection $f:\Pi'\injto \Pi$ such that $f^{-1}(\Delta)=\Delta'$ (where we note that in the case where $k$ is an NF, the assumption $f^{-1}(\Delta)=\Delta'$ is satisfied automatically by [\cite{SZM}, Theorem 5.1]). Then the following assertions hold:

    (i) $f$ induces an injection of $g:\Delta'\injto \Delta$ index $2$.

    (ii) $s'$ arises from a $k'$-rational point.

    (iii) The inclusion $\Delta'\injto \Pi'$ is constructible (cf. \cref{defi_constructible_inclusion}(ii)).

    (iv) The inclusion $\Delta\injto \Pi$ is constructible.

    (v) $s$ arises from a $k$-rational point.
\end{thmb}

This paper is organized as follows. In \cref{section_A Review and Generalization of Previous Results}, we review several classical anabelian results, as well as results developed in \cite{SZM}, and introduce various versions of the notion of an inclusion of \emph{CAVC-type}. In \cref{section_virtual_decuspidaloid}, we introduce the notion of a virtual decuspidaloid and examine associated group-theoretic constructions. In \cref{section_symmetry}, we study generic symmetries of configuration spaces. Finally, in \cref{section_reconstruction}, we carry out the reconstruction process. 


\section{Notations and Terminologies}\label{Notations and Terminologies}

~\\
\textbf{Groups and Topologies:}

Throughout the paper, homomorphisms of topological groups will be assumed, without further mention, to be continuous
(unless specified otherwise).


Let $S$ be a set. Then denote by $\Aut_{\set}(S)$ the group of automorphisms of $S$; denote by $\card(S)$ the cardinality of $S$.

Write $\{ 1 \}$ for the unit group.

Let $n$ be a positive integer. Then denote by $\mathbb{S}_{n}$ the permutation group on $n$ letters.

Let $G$ be a group that acts on a field $k$. Then denote by $k^G\subseteq k$ the subfield of elements fixed by $G$.

Let $G$ be a group; $H\subset G$ a subgroup of $G$. Then we say that $H$ is of \emph{index} $n$ (respectively, $H$ is of \emph{finite index}; $H$ is of \emph{infinite index}) in $G$ if the cardinality of the set of left cosets $G/H$ is $n$ (respectively, finite; infinite). 

Let $H$ be a subgroup of finite index of a group $G$. Then denote by $[G:H]$ the index of $H$ in $G$.

Let $f:H\to G$ be a group homomorphism. Then we say that $f$ is of \emph{index} $n$ (respectively, $f$ is of \emph{finite index}; $f$ is of \emph{infinite index}) if $f(H)$ is of index $n$ (respectively, $f(H)$ is of finite index; $f(H)$ is of infinite index) in $G$.

Let $H\subset G$ and $H'\subset G'$ be inclusions of groups such that $H$ (respectively, $H'$) is normal in $G$ (respectively, $G'$). Then a \emph{morphism of inclusions} from $H\subset G$ to $H'\subset G'$ is defined to be a group homomorphism $f:G\to G'$ such that $f^{-1}(H') = H$. 
Moreover, in this situation, denote by $f|_{H}$ the naturally induced group homomorphism from $H$ to $H'$.
For simplicity, we shall use the notation $f:G\to G'$ for a morphism of inclusions when there is no danger of confusion.

Let $G$ be a group and $S$ a set of subsets of $G$. Define $G/S$ to be the quotient group of $G$ by the normal subgroup generated by the union of all subsets in $S$, i.e.,
\begin{displaymath}
    G/S \coloneqq G / \langle \bigcup_{A \in S, g\in G} g\cdot A\cdot g^{-1} \rangle,
\end{displaymath}
where $\langle \bigcup_{A \in S, g\in G} g\cdot A\cdot g^{-1} \rangle$ denotes the subgroup of $G$ generated by the union of the $G$-conjugates of the subsets in $S$.  If $S$ is a set of sets of subsets of $G$, then we shall write $G/S$ for the quotient group $G/T$, where $T$ denotes the 
set of subsets of $G$ obtained by associating to each element of $S$ the union of the subsets of $G$ that constitute that element of $S$. 

Let $G$ be a topological group. 
Denote by $G^{\ab}$ the maximal abelian Hausdorff quotient of the topological group $G$. Denote by $G^{\abt}$ the maximal torsion-free Hausdorff quotient of $G^{\ab}$.
Note that if $G$ is profinite, then both $G^{\ab}$ and $G^{\abt}$ admit naturally induced profinite structures. If $f:G\rightarrow H$ is a homomorphism of profinite groups, we denote the naturally induced morphism by $f^{\abt}:G^{\abt}\rightarrow H^{\abt}$.

Let $G$ be a topological group; $A$ a topological $G$-module. Denote by $A_{G}$ the maximal Hausdorff quotient of $A$ on which $G$ acts trivially; denote by $A^{G}$ the topological submodule of $A$ consisting of the elements fixed by $G$.

Let $G$ be a topological group. Then denote by $\Aut(G)$ the group of automorphisms of the topological group $G$; denote by $\Inn(G)$ the group of inner automorphisms of $G$; denote by $\Out(G)$ the group of outer automorphisms of $G$ (i.e., the quotient $\Aut(G)/\Inn(G)$). Suppose further that $G$ is \emph{topologically finitely generated}. Then $G$ admits a topological basis of characteristic open subgroups, which induces natural profinite topologies on $\Aut(G)$ and $\Out(G)$.

Let $G$ be a group; $H$ a subgroup of $G$. Then denote by $N_{G}(H)$ the normalizer of $H$ in $G$.


~\\
\textbf{Fields, Schemes and Curves:}

Let $k$ be a field. Then we say that $k$ is an NF if $k$ is a number field; 
we say that $k$ is an MLF if $k$ is a mixed-characteristic local field; 
we say that $k$ is an AF (arithmetic field) if $k$ is either an NF or an MLF.

Let $k$ be a field; $G$ a group acting faithfully on $k$. 
Let $X$ be a scheme over $k$ equipped with a $G$-action compatible with the $G$-action on $k$. 
Suppose that $X$ is covered by affine open subschemes $U$ that are stabilized by the action of 
$G$.  Then denote by $X^{G}$ the scheme obtained from $X$ by gluing together the 
rings of $G$-invariants of $\Gamma(U,\mathcal{O}_U)$ for such affine open subschemes $U$. 

Let $k$ be a topological field (such as a field equipped with the discrete topology). 
Then denote by $\Aut(k)$ the group of continuous automorphisms of the topological 
field $k$.

Let $X$ be an integral scheme. Then denote by $K(X)$ the function field of $X$.

Let $X$ and $Y$ be smooth curves. Then we say that a morphism $f:X\to Y$ is a decuspidalization if $f$ is an open immersion.

Let $X$ be a smooth curve over a field $k$; $\bar{k}$ an algebraic closure of $k$. 
Then denote by $X^{\cl}$ the (unique) smooth compactification of $X$. Write $\bar{X}\coloneqq X\times_{k} \bar{k}$, $\bar{X}^{\cl}\coloneqq X^{\cl}\times_{k} \bar{k}$. Write $g(X)$ for the genus of $\bar{X}^{\cl}$; $r(X)$ for the cardinality of the set of points in $\bar{X}^{\cl}\setminus\bar{X}$ (i.e., the set of \emph{geometric points} of $X^{\cl}\setminus X$). We refer to the pair $(g(X),r(X))$ as the \emph{type} of $X$. We say that a smooth curve is \emph{hyperbolic} if it is of type $(g,r)$, where $2g+r\geq 3$.

Let $Y$ be an abelian variety over a field $k$; $f:Y\simto Y$ an automorphism of the $k$-scheme 
$Y$. Then we say that $f$ is a translation of $Y$ if there exists a point $y\in Y(k)$ such that $f$ is the morphism $x\mapsto x+y$. We remark that if one endows $Y$ with another group scheme structure (possibly having a different identity 
element), then the resulting set of translations, viewed as a subset of $\Aut(Y)$, coincides 
with the original one [\cite{Mumf}, Chapter II, \S 4, Corollary 1]. Therefore, the notion of 
a translation is well-defined without reference to a particular choice of group structure on $Y$.

We refer to [\cite{MZK9}, \S 0, the discussion entitled ``Curves''] for the definition of a family of hyperbolic curves of type $(g,r)$. The scheme denoted by ``$Y$'' in the cited discussion shall be regarded as the compactification of the family of curves under consideration. Moreover, we note that, as stated in the same discussion, if the base scheme is normal, then the compactification of a family of hyperbolic curves of type $(g,r)$ is unique up to canonical isomorphism. 

Let $p$ be either a prime number or $0$. Denote by $\bar{\mathbb{Q}}_{(p)}$ the following \emph{topological} field:
\begin{itemize}
    \item the algebraic closure of $\mathbb{Q}_{p}$, equipped with the topology induced by $\mathbb{Q}_{p}$, if $p$ is a prime number;
    \item the algebraic closure of $\mathbb{Q}$, equipped with trivial topology, if $p = 0$.
\end{itemize}
Moreover, we refer to such fields as \emph{algebraically closed arithmetic fields} (abbreviated as \emph{ACAF}). 
Let $K$ be an ACAF. Then we define the \emph{characteristic prime} of $K$ to be the unique integer $p$ such that 
there exists an isomorphism of topological fields $K\simto \bar{\mathbb{Q}}_{(p)}$. We remark that the {\it minimal 
closed subfield} of $K$ exists and is naturally isomorphic to $\mathbb{Q}$ (respectively, $\mathbb{Q}_{p}$) if the characteristic prime 
of $K$ is zero (respectively, a prime number).

Let $K$ be an ACAF; $k$ the minimal closed subfield of $K$; $G$ a profinite group; $\Aut(K)$ the group of 
field automorphisms of the topological field $K$; $f$ a group homomorphism $G\to \Aut(K)$ of abstract groups 
(i.e., disregarding the topology on $G$). Note that $\Aut(K)$ is naturally isomorphic, as an abstract group, to 
$\gal(K/k)$. Equip $\Aut(K)$ with the natural topology (i.e., the Krull topology) of $\gal(K/k)$. We say that $f$ is 
\emph{field-theoretic} (abbreviated as FT) if $f$ is of finite index and determines a homeomorphism of $G$ onto its 
image in $\Aut(K)$.
Let $k'$ denote the finite extension of $k$ fixed by the image of $f$. We refer to $k'$ as the \emph{fixed field} of $f$.

Let $G$ be a profinite group; $K$ an ACAF; $f:G\to \Aut(K)$, $f':G\to \Aut(K)$ FT morphisms. Then we say that $f$ and $f'$ 
are {\it conjugate} if there exists an \emph{inner} automorphism $\sigma:\Aut(K)\simto \Aut(K)$ of $\Aut(K)$ such that $f' = \sigma\circ f$.

~\\
\textbf{Virtual Curves:}

We refer to [\cite{SZM}, Definition 3.1; 
\cite{SZM}, Definition 3.5] for the definitions of virtual varieties, pointed virtual curves, virtual fundamental groups, and geometric virtual fundamental groups. Moreover, we observe that the definition of a family of hyperbolic curves of type $(g,r)$ given in 
\cite{SZM} differs, at first glance, from the definition adopted in the present paper, but one verifies immediately that 
the two definitions are equivalent in the case where the base scheme is reduced and connected. 
Moreover, [\cite{MZK9}, \S \, 0, the discussion entitled ``Curves''] implies that the object denoted by ``$Z$'' in 
[\cite{SZM}, Proposition 5.3] is unique up to canonical isomorphism.

We refer to [\cite{SZM}, Proposition 5.3; \cite{SZM}, Definition 5.4] for the definitions of the notion of a \emph{decuspidalization} of 
pointed virtual curves and the notion of the \emph{conjugacy class} determined by such a decuspidalization. Furthermore, we observe that in [\cite{SZM}, Definition 5.4(ii)], the notion of the conjugacy class determined by a decuspidalization of pointed virtual curves admits a natural generalization to decuspidalizations of arbitrary degree, where the unique conjugacy class is replaced by a subset of the set of conjugacy classes of cuspidal inertia subgroups. If this subset has cardinality $1$, then we may still refer to its unique element as the conjugacy class determined by the respective decuspidalization.

Let $\mathcal{C}:(X,Y,f,k,\bar{k},t,s)$ be a virtual variety. Suppose that $h:X\to X$ is a scheme automorphism satisfying $f\circ h = f$. Then $\mathcal{F}\coloneqq(h,\id_{Y},\id_{k})$, where $\id_{Y}$ (respectively, $\id_{k}$) denotes the identity on $Y$ (respectively, $k$), defines an automorphism of the virtual variety $\mathcal{C}$. We say that $\mathcal{F}$ is induced by $h$. For convenience, when no confusion arises, we shall also denote this automorphism by $h$.

~\\
\textbf{Fundamental Groups:}

For a connected scheme $X$ equipped with a geometric point $\bar{x}\rightarrow X$, denote by $\pi_{1}(X,\bar{x})$ the \etale fundamental group of $X$ with basepoint $\bar{x}$.

Throughout the paper, we omit the basepoint since there exists a natural isomorphism between fundamental groups associated with different basepoints, well-defined up to an inner automorphism. We denote the \'{e}tale fundamental group of $X$ with basepoint omitted as $\Pi_{X}$. If $X$ is the spectrum of a field $k$, we also use the notation $G_{k}$ (i.e., the ``absolute Galois group of $k$'') to denote $\Pi_{X}$. Let $X$ be a scheme that is geometrically connected over a field $k$. Then denote by $\Delta_{X}$ the kernel of the natural surjection $\Pi_{X}\surjto G_{k}$, i.e., the \emph{geometric fundamental group}.

Any morphism $f:X\rightarrow Y$ of connected schemes induces an outer homomorphism $\Pi_{X}\rightarrow \Pi_{Y}$, which we denote by $f_{\ast}$.

Let $k$ be a field of characteristic zero; $X$ a smooth hyperbolic curve over $k$. Then denote by $\cusp(X)$ the set of conjugacy classes of cuspidal inertia subgroups of $\Pi_{X}$, i.e., the intersections of decomposition groups of cusps with $\Delta_{X}$, the geometric fundamental group of $X$.

Let $X$ be an integral scheme over a field of characteristic zero. Define the \emph{function field of the universal pro-\etale covering} of $X$, abbreviated as the \emph{FFUC} of $X$, to be the direct limit of the function fields $K(Y)$, where $Y$ ranges over all finite \etale coverings of $X$.


Let $f:X\to S$ be a family of hyperbolic curves; $X\injto Y\to S$ a compactification; $D\coloneqq Y\setminus X$ the relative divisor of $Y$ induced by $X$. Assume further that $S$ is a Noetherian separated normal connected scheme over a field of characteristic zero. 
Note that an \etale covering $h:X'\to X$ naturally induces a family $X'\to S'$ of hyperbolic curves, together with an \etale covering $i:S'\to S$ (cf. [\cite{Hos3}, Proposition 2.3]). 
Let $X'\injto Y'\to S'$ be a compactification. 
Then $f$ induces a proper morphism $h:Y'\to Y$ such that $h^{-1}(X)=X'$ (cf. [\cite{Hos3}, Proposition 2.3]). 
Define the \emph{set of geometric cusps} associated to the family $f$ of hyperbolic curves to be the inverse limit $\tilde{D}$ of the sets of the irreducible components of the inverse images of $D$ in the (unique) compactifications of various finite subcoverings of the universal profinite \etale covering of $X$ corresponding to $\Pi_{X}$.
Let $x\in \tilde{D}$. Then the \emph{decomposition group} associated to $x$ of $\Pi_{X}$ is defined to be the stabilizer of $x\in \tilde{D}$ with respect to the natural action of $\Pi_{X}$ on $\tilde{D}$. The \emph{inertia group} associated to $x$ of $\Pi_{X}$ is defined to be the kernel of the natural map from the decomposition group associated to $x$ of $\Pi_{X}$ to $\Pi_{S}$.
Denote by $D(f, \Pi_{X})$ the set of decomposition groups $\subseteq\Pi_{X}$ associated to elements $\in\tilde{D}$; $I(f, \Pi_{X})$ the set of inertia groups $\subseteq\Pi_{X}$ associated to elements $\in\tilde{D}$.
Write $\Delta$ for the kernel of $\Pi_{X}\to \Pi_{S}$. Note that elements of $I(f, \Pi_{X})$ are subgroups of $\Delta$. Define the \emph{set of conjugacy classes of cuspidal inertia groups} associated to the family $f$ of hyperbolic curves to be the set of $\Delta$-conjugacy classes of elements of $I(f, \Pi_{X})$. Denote the set of conjugacy classes of cuspidal inertia groups associated to the family $f$ of hyperbolic curves by $\cusp(f, \Pi_{X})$.

Throughout the paper, when we consider various morphisms and constructions of groups that arise from fundamental groups, 
we shall assume, without further explanation, that those morphisms and constructions are only defined up to an inner 
automorphism.  On the other hand, when we consider groups that may be isomorphic to, but \emph{do not directly arise 
from fundamental groups} (such as the groups that appear in the inclusions of ``CAVC-type'' and ``cls-CAVC-type'' 
discussed in \cref{section_virtual_decuspidaloid}), we work with homomorphisms of groups, rather than homomorphisms 
regarded up to composition with an inner automorphism.

~\\
\textbf{Graphs and Categories:}

Throughout the paper, categories are assumed to be small.

Let $\mf{X}$ be a category. Then we say that a tuple $(\mf{A},A)$ is an \emph{$\mf{X}$-structured category} if $\mf{A}$ is a category and $A: \mf{A}\to \mf{X}$ is a functor. One may say that $(\mf{A},A)$ is a \emph{structured category} when there is no danger of confusion.

Let $\mf{X}$ be a category; $(\mf{A},A)$, $(\mf{B},B)$ two $\mf{X}$-structured categories. We say that a tuple $(F, \tau)$ is a \emph{morphism} from $(\mf{A},A)$ to $(\mf{B},B)$ if $F:\mf{A}\to \mf{B}$ is a functor and $\tau$ is a natural transformation
    \begin{displaymath}
        \tau: A\Rightarrow B\circ F.
    \end{displaymath}

Let $\mf{X}$ be a category; $(\mf{A},A)$, $(\mf{B},B)$ two $\mf{X}$-structured categories; $(F, \tau): (\mf{A},A)\to (\mf{B},B)$ a morphism. We say that $(F, \tau)$ is an \emph{immersion} if $F$ is \emph{injective} on both the set of objects and the set of morphisms, and $\tau$ is a natural isomorphism. We say that $(F, \tau)$ is an \emph{isomorphism} if $F$ is \emph{bijective} on both the set of objects and the set of morphisms, and $\tau$ is a natural isomorphism. We say that $(F, \tau)$ is an \emph{automorphism} if $\mf{A} = \mf{B}$, $A = B$, and $(F, \tau)$ is an isomorphism. 

We remark that the definitions of structured categories are \emph{not} taken up to equivalence of categories.

~\\
\textbf{Configuration Spaces:}

Let $X$ be a variety over a field $k$. Define the \emph{2-configuration space} $X_{2}$ of $X$ to be the complement of the diagonal in $X\times_{k} X$. The first (respectively, second) projection $\pr^{1}_{X}$ (respectively, $\pr^{2}_{X}$) from $X\times_{k} X$ to $X$ restricts naturally to $X_{2}$. Denote the restriction of $\pr^{1}_{X}$ (respectively, $\pr^{2}_{X}$) to $X_2$ by $\pr^{2/1}_{X}$ (respectively, $\pr^{1\backslash 2}_{X}$).

\section{Review and Generalization of Previous Results}\label{section_A Review and Generalization of Previous Results}

\begin{defi}\label{defi_CAVCs} \quad

    (i) We say that a profinite group $G$ is of \emph{semi-cuspidally-admissible-virtual-curve-type}, abbreviated as \emph{semi-CAVC-type}, if there exists a pointed virtual curve $\mathcal{C}:(X,Y,f,k,\bar{k},t,s)$ of type $(g,r)$ and a prime number $l$ such that $r\geq 1$; $k$ is $l$-cyclotomically full (cf. [\cite{SZM}, Definition 4.10(ii)]); $Y$ is smooth over $k$; and there exists an isomorphism $\phi:G\simto \Pi_{\mathcal{C}}$. Moreover, we say that $(\mathcal{C},\phi)$ as above is a collection of \emph{semi-virtual-data}, abbreviated as \emph{semi-VD}, of $G$.

    (ii) We say that an inclusion of profinite groups $G_{\Delta}\subset G_{\Pi}$ is of \emph{cuspidally-admissible-virtual-curve-type}, abbreviated as \emph{CAVC-type}, if there exists a pointed virtual curve $\mathcal{C}:(X,Y,f,k,\bar{k},t,s)$ of type $(g,r)$ and a prime number $l$ such that $r\geq 1$; $k$ is $l$-cyclotomically full; $Y$ is smooth over $k$; and there exists an isomorphism $\phi:G_{\Pi}\simto \Pi_{\mathcal{C}}$ such that $\phi(G_{\Delta}) = \Delta_{\mathcal{C}}$. Moreover, we say that $(\mathcal{C},\phi)$ as above is a collection of \emph{virtual-data}, abbreviated as \emph{VD}, of $G_{\Delta}\subset G_{\Pi}$.

    (iii) We say that an inclusion of profinite groups $G_{\Delta}\subset G_{\Pi}$ is of \emph{closed-cuspidally-admissible-virtual-curve-type}, abbreviated as \emph{cls-CAVC-type}, if there exists a pointed virtual curve $\mathcal{C}:(X,Y,f,k,\bar{k},t,s)$ of type $(g,0)$ and a prime number $l$ such that $k$ is $l$-cyclotomically full; $Y$ is smooth over $k$; and there exists an isomorphism $\phi:G_{\Pi}\simto \Pi_{\mathcal{C}}$ such that $\phi(G_{\Delta}) = \Delta_{\mathcal{C}}$. Moreover, we say that $(\mathcal{C},\phi)$ as above is a collection of \emph{closed-virtual-data}, abbreviated as \emph{cls-VD}, of $G_{\Delta}\subset G_{\Pi}$.
\end{defi}

\begin{coro}\label{CAVC_universal_prop}
    Let $G_{\Delta}\subset G_{\Pi}$ be an inclusion of CAVC-type. Let $(\mathcal{C},\phi)$, $(\mathcal{C}',\phi')$ be VDs of $G_{\Delta}\subset G_{\Pi}$. Then the types of $\mathcal{C}$ and $\mathcal{C}'$ coincide. Moreover, $\phi'\circ \phi^{-1}$ induces an isomorphism $\cusp(\mathcal{C}) \simto \cusp(\mathcal{C}')$.
\end{coro}
\begin{proof}
    By [\cite{SZM}, Proposition 4.18(i); \cite{SZM}, Theorem 4.20(ii)], one may group-theoretically reconstruct the type of $\mathcal{C}$ (respectively, $\mathcal{C}'$) as well as $\cusp(\mathcal{C})$ (respectively, $\cusp(\mathcal{C}')$) from the inclusion $\Delta_{\mathcal{C}}\injto \Pi_{\mathcal{C}}$. Therefore, \cref{CAVC_universal_prop} follows immediately.
\end{proof}

\begin{defi}\label{defi_cusp_33}
    Let $G_{\Delta}\subset G_{\Pi}$ be an inclusion of CAVC-type.
    Then we say that $G_{\Delta}\subset G_{\Pi}$ is of \emph{type} $(g,r)$ if there exists a VD $(\mathcal{C},\phi)$ such that $\mathcal{C}$ is of type $(g,r)$. Moreover, we define $\cusp(G_{\Delta}\subset G_{\Pi})$ to be the set of $G_{\Delta}$-conjugacy classes of closed subgroups of $G_{\Delta}$ that corresponds to $\cusp(\mathcal{C})$ under $\phi$ for some VD $(\mathcal{C},\phi)$. 
    Note that by \cref{CAVC_universal_prop}, these definitions do not depend on the choices of the VD.
    For simplicity, if there is no danger of confusion with respect to $G_{\Delta}$, we define the \emph{type} of $G_{\Pi}$ (respectively, write $\cusp(G_{\Pi})$) to be the type of $G_{\Delta}\injto G_{\Pi}$ (respectively, for $\cusp(G_{\Delta}\injto G_{\Pi})$).
\end{defi}

\begin{prop}\label{cls_diff_uncls}
    Let $G_{\Delta} \subset G_{\Pi}$ be an inclusion of cls-CAVC-type. Then the following assertions hold:

    (i) Let $(\mathcal{C},\phi)$, $(\mathcal{C}',\phi')$ be cls-VDs of $G_{\Delta}\subset G_{\Pi}$. Then the types of $\mathcal{C}$ and $\mathcal{C}'$ coincide.
    
    (ii) Let $G_{\Delta}' \subset G_{\Pi}'$ be an inclusion of CAVC-type. Then there does not exist an isomorphism $\phi: G_{\Pi} \simto G_{\Pi}'$ such that $\phi(G_{\Delta}) = G_{\Delta}'$.
\end{prop}
\begin{proof}
    Assertion (i) follows immediately from the fact that the genus of a pointed virtual curve (under the assumption that the cardinality of the set of the cusps is zero) can be group-theoretically reconstructed from the rank of the abelianization of its geometric virtual fundamental group (cf. [\cite{SZM}, Proposition 3.7; \cite{MZK1}, Remark 1.2.2]).
    Assertion (ii) follows immediately from the fact that $G_{\Delta}$ and $G_{\Delta}'$ have different cohomological dimensions (cf. [\cite{SZM}, Proposition 3.7; \cite{MZK1}, Remark 1.2.2]).
\end{proof}

\begin{defi}
    Let $G_{\Delta}\subset G_{\Pi}$ be an inclusion of cls-CAVC-type.
    Then we say that $G_{\Delta}\subset G_{\Pi}$ is of \emph{type} $(g,0)$ if there exists a cls-VD $(\mathcal{C},\phi)$ such that $\mathcal{C}$ is of type $(g,0)$.
    Note that by \cref{cls_diff_uncls}(i), this definition does not depend on the choice of the cls-VD. For simplicity, if there is no danger of confusion with respect to $G_{\Delta}$, we define the \emph{type} of $G_{\Pi}$ to be the type of $G_{\Delta}\injto G_{\Pi}$.
\end{defi}

\begin{defi}
    We say that an inclusion of profinite groups $G_{\Delta}\subset G_{\Pi}$ is of \emph{general-cuspidal-admissible-virtual-curve-type}, abbreviated as \emph{gen-CAVC-type}, if it is of either CAVC-type or cls-CAVC-type.
\end{defi}

\begin{remm}\label{remm_37}
    By \cref{cls_diff_uncls}(ii), for any inclusion $G_{\Delta}\subset G_{\Pi}$ of gen-CAVC-type, the notion of the type
    $(g,r)$ of the inclusion $G_{\Delta}\subset G_{\Pi}$ is well-defined. Moreover, by defining the set of conjugacy classes of cuspidal inertia groups to be the empty set in the case where $G_{\Delta} \subset G_{\Pi}$ is of cls-CAVC-type, we obtain a 
    well-defined notion of $\cusp(G_{\Delta} \subset G_{\Pi})$.
\end{remm}

\begin{defi}\label{defi_admissible_type}
    Let $G_{\Delta}\subset G_{\Pi}$ be an inclusion of CAVC-type of type $(g,r)$; $S$ a subset of $\cusp(G_{\Delta}\subset G_{\Pi})$. We say that $S$ is of \emph{admissible-type} if $S$ is stabilized by the natural action of $G_{\Pi}$ by conjugation, and one of the following holds:
    \begin{enumerate}
        \item $g \geq 2$;
        \item $g = 1$ and $\card(S) \leq r - 1$;
        \item $g = 0$ and $\card(S) \leq r - 3$.
    \end{enumerate}
    We say that $S$ is of \emph{CAVC-admissible-type} if $S$ is stabilized by the natural action of $G_{\Pi}$ by conjugation, and one of the following holds:
    \begin{enumerate}
        \item $g \geq 1$ and $\card(S) \leq r - 1$;
        \item $g = 0$ and $\card(S) \leq r-3$.
    \end{enumerate}
\end{defi}

\begin{lemm}\label{field_ext_curve_theoretic}
    Let $l$ be a prime number; $k$ an $l$-cyclotomically full field of characteristic zero; $\mathcal{C}:(X,Y,f,k,\bar{k},t,s)$ a pointed virtual curve of type $(g,r)$ such that $Y$ is smooth over $k$. Let $\Pi$ be an open subgroup of $\Pi_{\mathcal{C}}$ containing $\Delta_{\mathcal{C}}$. Denote by $h:\sp(k')\to \sp(k)$ the morphism induced by the finite field extension $k'/k$ determined by the image of $\Pi$ in $G_k$ via the natural surjection $\Pi_{\mathcal{C}}\surjto G_{k}$; by $s'$ the section $G_{k'}\injto \Pi_{Y\times_{k}k'}$ induced by $s$; by $t'$ the natural morphism $Y\times_{k}k'\to \sp(k')$. Then 
    \begin{displaymath}
    F\coloneqq(X\times_{\sp(k)}h,Y\times_{\sp(k)}h,h):\mathcal{C'}\coloneqq(X\times_{k}k',Y\times_{k}k',f\times_{k}k',k',\bar{k},t',s')\to (X,Y,f,k,\bar{k},t,s)
    \end{displaymath}
    is an \etale covering (cf. [\cite{SZM}, Definition 3.3]), and the natural image of $\Pi_{\mathcal{C}'}$ in $\Pi_{\mathcal{C}}$ is $\Pi$. Moreover, $F$ induces a natural isomorphism $\cusp(\mathcal{C'})\simto \cusp(\mathcal{C})$.
\end{lemm}
\begin{proof}
    It follows immediately from the constructions involved that $F$ is an \etale covering, and 
    that thel image of $\Pi_{\mathcal{C}'}$ in $\Pi_{\mathcal{C}}$ is $\Pi$. It follows immediately from [\cite{SZM}, Theorem 4.20(ii)] that $F$ induces a natural isomorphism $\cusp(\mathcal{C'})\simto \cusp(\mathcal{C})$.
\end{proof}

\begin{defi}\label{defi_cusp_compatibility_subgroup}
    Let $G$, $H$ be topological groups; $S$ a set of conjugacy classes of commensurably terminal (cf. [\cite{MZK3}, \S \, 0, the discussion entitled ``Topological Groups'']) subgroups of $G$; $f:H\injto G$ an open injective homomorphism of finite index. Define
    \begin{displaymath}
    T \coloneqq \{ \beta \cap H\ |\ \beta \in \alpha \in S \},
    \end{displaymath}
    which is a collection of subgroups of $H$. By construction, $T$ is stable under the natural conjugation action of $H$, if one regards $T$ as a subset of the set of subgroups of $H$. Consequently, $T$ determines a set of conjugacy classes of subgroups of $H$, which we denote by $S'$. 
    By considering the commensurators, we obtain a natural surjection $h:S'\surjto S$. We refer to $S'$ as the \emph{set of conjugacy classes} of $H$ induced by $S$ via $f$, and to $h$ as the \emph{induced morphism of sets of conjugacy classes}.
\end{defi}

\begin{lemm}\label{covering_curve_theoretic}
Let $l$ be a prime number; $k$ an $l$-cyclotomically full field of characteristic zero; $\mathcal{C}:(X,Y,f,k,\bar{k},t,s)$ a pointed virtual curve of type $(g,r)$ such that $Y$ is smooth over $k$. Let $\Pi$ be an open subgroup of $\Pi_{\mathcal{C}}$. Then there exists an \etale covering $\mathcal{C}' \to \mathcal{C}$ of pointed virtual curves such that the natural image of $\Pi_{\mathcal{C}'}$ in $\Pi_{\mathcal{C}}$ is $\Pi$.
\end{lemm}
\begin{proof}
    The assertion follows immediately from [\cite{SZM}, Corollary 3.4(ii)].
\end{proof}

\begin{coro}\label{covering_group_theoretic}
    Let $G_{\Delta}\subset G_{\Pi}$, $G_{\Delta}'\subset G_{\Pi}'$ be two inclusions of profinite groups; $f:G_{\Pi}'\to G_{\Pi}$ a morphism of inclusions such that $f$ is an open injective homomorphism of finite index. Assume that $G_{\Delta}\subset G_{\Pi}$ is of CAVC-type (respectively, cls-CAVC-type). Then $G_{\Delta}'\subset G_{\Pi}'$ is of CAVC-type (respectively, cls-CAVC-type). Moreover, $\Cusp(G_{\Pi}')$ is the set of conjugacy classes of $G_{\Delta}'$ of subgroups induced by $\Cusp(G_{\Pi})$ via the restriction $f|_{G_{\Delta}'}$.
\end{coro}
\begin{proof}
    It follows immediately from \cref{covering_curve_theoretic} that $G_{\Delta}'\subset G_{\Pi}'$ is of CAVC-type (respectively, cls-CAVC-type). By \cref{field_ext_curve_theoretic} and [\cite{SZM}, Theorem 4.20(ii)], one may assume without loss of generality that $G_{\Delta}\subset G_{\Pi}$ arises from a pointed virtual curve whose section arises from a rational point. The remaining assertion then follows from [\cite{SZM}, Theorem 4.20(iii)] and the well-known structure of the \etale fundamental groups of hyperbolic curves.
\end{proof}

\begin{lemm}\label{Galois_actions_on_D}
    Let $l$ be a prime number; $k$ an $l$-cyclotomically full field of characteristic zero; $\mathcal{C} : (X,Y,f,k,\bar{k},t,s)$ a pointed virtual curve of type $(g,r)$ such that $Y$ is smooth over $k$ and $r\geq 1$; $S \subset \cusp(\mathcal{C})$ a subset of admissible-type that is stabilized by the natural conjugation action of $\Pi_{X}$.
    Then there exists a decuspidalization $\mathcal{C} \to \mathcal{C}'$ of pointed virtual curves 
    such that $S$ is the subset determined by the decuspidalization 
    (cf. \cref{Notations and Terminologies}, the discussion entitled ``Virtual Curves'').
\end{lemm}
\begin{proof}
    Let $\xi$ be the generic point of $Y$; $\bar{\xi}$ a geometric point over $\xi$; $f' \colon Z \to Y$ the unique smooth compactification of $X$ over $Y$ (cf. \cref{Notations and Terminologies}, the discussion entitled ``Virtual Curves''). Write $X_{\xi}\coloneqq X\times_{Y} \xi$; $X_{\bar{\xi}}\coloneqq X\times_{Y} \bar{\xi}$; $Z_{\xi}\coloneqq Z\times_{Y} \xi$. Consider the following commutative diagram:
    \begin{displaymath}
	\xymatrix{
		\Pi_{X_{\bar{\xi}}}\ar[r]\ar[d] & \Pi_{X_{\xi}}\ar[r]\ar[d] & G_{K(Y)}\ar[d] \\
		\Delta_{\mathcal{C}} \ar[r] & \Pi_{X} \ar[r] & \Pi_{Y}.	
	}
    \end{displaymath}
    By [\cite{SZM}, Proposition 3.7], the left-hand vertical arrow is an isomorphism. Hence, one may regard $\cusp(\mathcal{C})$ as a set of conjugacy classes of subgroups of $\Delta_{X_{\xi}} \coloneqq \Pi_{X_{\bar{\xi}}}$. 
    It follows from a similar argument to the argument applied in the proof of [\cite{SZM}, Theorem 4.20(iii)] that $\cusp(\mathcal{C}) = \cusp(X_{\xi})$ as a set of conjugacy classes of subgroups of $\Delta_{X_{\xi}}$.
    By assumption, $S$ is stabilized by the action of $G_{K(Y)}$. Thus $S$ naturally corresponds to a set of closed points of $Z_{\xi}$ that lie in the complement $Z_{\xi} \setminus X_{\xi}$.
    By construction, such a set naturally corresponds to a set of closed irreducible subschemes of $Z$ that lie in the complement $Z\setminus X$ and dominate $Y$. Let $D$ be the union of these closed irreducible subschemes; $X'$ the complement $Z \setminus D$; $f'' \colon X' \to Y$ the morphism induced from $f'$. 
    By construction, $\mathcal{C}' : (X', Y, f'', k, \bar{k}, t, s)$ defines a pointed virtual curve, and the morphism $\mathcal{C} \to \mathcal{C}'$ induced by the open immersion $X \injto X'$ is a decuspidalization of pointed virtual curves.
\end{proof}

\begin{lemm}\label{decuspidalization_curve_theoretic}
Let $l$ be a prime number; $k$ an $l$-cyclotomically full field of characteristic zero; $\mathcal{C}:(X,Y,f,k,\bar{k},t,s)$ a pointed virtual curve of type $(g,r)$ such that $Y$ is smooth over $k$ and $r\geq 1$; $S \subset \cusp(\mathcal{C})$ a subset of admissible-type. Then there exists an \etale covering 
\begin{displaymath}
(h,i,j):\mathcal{C}'\coloneqq(X',Y',f',k',\bar{k}',t',s')\to \mathcal{C}
\end{displaymath}
satisfying the following properties:
\begin{enumerate}
    \item The morphism $(h,i,j)$ induces an isomorphism $\Pi_{\mathcal{C}'} \simto \Pi_{\mathcal{C}}$, and hence also an isomorphism $\Delta_{\mathcal{C}'} \simto \Delta_{\mathcal{C}}$.
    \item Recall that $(h,i,j)$ induces a natural bijection $\cusp(\mathcal{C}') \simto \cusp(\mathcal{C})$ (cf. (1),  [\cite{SZM}, Theorem 4.20(ii)]). Thus one may regard $S$ as a subset of $\cusp(\mathcal{C}')$. Then there exists a decuspidalization $\mathcal{C}' \to \mathcal{C}''$ that determines $S$.
\end{enumerate}
\end{lemm}
\begin{proof}
    By [\cite{SZM}, Corollary 3.2; \cite{SZM}, Theorem 4.20(ii)], $\Pi_{Y}$ acts naturally on $\cusp(\mathcal{C})$. Let $H\subset \Pi_{Y}$ be the stabilizer of $S$ in $\Pi_{Y}$. Then it is immediate that $H$ is open in $\Pi_{Y}$. Write $c:Y'\to Y$ for the \etale covering of $Y$ determined by $H$. Since $s(G_k)$ stabilizes $S$, it follows that $Y'$ is defined over $k$ and that $s(G_k) \subset H$. Write $s':G_{k}\to H = \Pi_{Y'}$ for the section determined by $s$.
    By construction,
    \begin{displaymath} (X\times_{Y}c,c,\id):\mathcal{C}'\coloneqq(X\times_{Y}Y',Y',f\times_{Y}Y',k,\bar{k},t\circ c,s')\to \mathcal{C}
    \end{displaymath}
    is an \etale covering inducing an isomorphism $\Pi_{\mathcal{C}'}\simto \Pi_{\mathcal{C}}$. Therefore, without loss of generality, we may assume that the natural action of $\Pi_Y$ on $\cusp(\mathcal{C})$ stabilizes $S$. The assertion then follows immediately from \cref{Galois_actions_on_D}.
\end{proof}

\begin{prop}\label{decuspidalization-group-theoretic}
    Let $G_{\Delta}\subset G_{\Pi}$ be an inclusion of CAVC-type. Let $S \subset \cusp(G_{\Pi})$ be of admissible-type (respectively, CAVC-admissible-type). Then the inclusion $G_{\Delta}/S\subset G_{\Pi}/S$ (cf. \cref{Notations and Terminologies}, the discussion entitled ``Groups and Topologies'') is of gen-CAVC-type (respectively, CAVC-type).
\end{prop}
\begin{proof}
    By \cref{decuspidalization_curve_theoretic}, there exists a VD (cf. \cref{defi_CAVCs}(ii)) $(\mathcal{C},\phi)$ of $G_{\Delta}\subset G_{\Pi}$, together with a decuspidalization $h:\mathcal{C}\to \mathcal{C}'$ of pointed virtual curves that determines the subset of $\cusp(\mathcal{C})$ 
    given by the natural image of $S$. By construction, there exists an isomorphism $i:G_{\Pi}/S\to \Pi_{\mathcal{C}'}$ such that $i(G_{\Delta}/S) = \Delta_{\mathcal{C}'}$. 
    If $S$ is of CAVC-admissible-type, then $\mathcal{C}'$ is of type $(g,r')$ with $r'\geq 1$ and $2g+r'\geq 3$; hence $G_{\Delta}/S\subset G_{\Pi}/S$ is of CAVC-type. If $S$ is of admissible-type but not of CAVC-admissible-type, then $\mathcal{C}'$ is of type $(g,0)$ with $g\geq 2$; hence $G_{\Delta}/S\subset G_{\Pi}/S$ is of cls-CAVC-type.
\end{proof}

\begin{defi}\label{defi_CAVC_morphism_type}\quad
    (i) Let $G_{\Delta}\subset G_{\Pi}$, $G_{\Delta}'\subset G_{\Pi}'$ be inclusions of gen-CAVC-type. Then we say that a morphism $f:G_{\Pi}\to G_{\Pi}'$ of inclusions is of \emph{geometric-isomorphism-type}, abbreviated as \emph{geo-iso-type}, if $f$ is injective, $f(G_{\Pi})$ is open in $G_{\Pi}'$, and $f(G_{\Delta})= G_{\Delta}'$.

    (ii) Let $G_{\Delta}\subset G_{\Pi}$, $G_{\Delta}'\subset G_{\Pi}'$ be inclusions of gen-CAVC-type. Then we say that a morphism $f:G_{\Pi}\to G_{\Pi}'$ of inclusions is of \emph{covering-type}, abbreviated as \emph{cov-type}, if $f$ is injective, and $f(G_{\Pi})$ is open in $G_{\Pi}'$.

    (iii) Let $G_{\Delta}\subset G_{\Pi}$ be an inclusion of CAVC-type; $G_{\Delta}'\subset G_{\Pi}'$ an inclusion of gen-CAVC-type. Then we say that a morphism $f:G_{\Pi}\to G_{\Pi}'$ of inclusions is of \emph{decuspidalization-type}, abbreviated as \emph{decusp-type}, if $f$ is surjective, and there exists a subset $S\subset \cusp(G_{\Pi})$ of admissible-type such that the kernel of $f$ is equal to the closed normal subgroup generated by the union of the subgroups in $S$.

    (iv) Let $G_{\Delta}\subset G_{\Pi}$ be an inclusion of CAVC-type; $G_{\Delta}'\subset G_{\Pi}'$ an inclusion of gen-CAVC-type. Then we say that a morphism $f:G_{\Pi}\to G_{\Pi}'$ of inclusions is of \emph{geometric-decuspidalization-type}, abbreviated as \emph{geo-decusp-type}, if $f(G_{\Pi})$ is open in $G_{\Pi}'$, $f(G_{\Delta})= G_{\Delta}'$, and there exists a subset $S\subset \cusp(G_{\Pi})$ of admissible-type such that the kernel of $f$ is equal to the closed normal subgroup generated by the union of the subgroups in $S$.

    (v) Let $G_{\Delta}\subset G_{\Pi}$ be an inclusion of CAVC-type; $G_{\Delta}'\subset G_{\Pi}'$ an inclusion of gen-CAVC-type. Then we say that a morphism $f:G_{\Pi}\to G_{\Pi}'$ of inclusions is of \emph{admissible-type} if $f(G_{\Pi})$ is open in $G_{\Pi}'$, and there exists a subset $S\subset \cusp(G_{\Pi})$ of admissible-type such that the kernel of $f$ is equal to the closed normal subgroup generated by the union of the subgroups in $S$.

    (vi) Let $G_{\Delta}\subset G_{\Pi}$ be an inclusion of cls-CAVC-type; $G_{\Delta}'\subset G_{\Pi}'$ an inclusion of gen-CAVC-type. Then we say that a morphism $f:G_{\Pi}\to G_{\Pi}'$ of inclusions is of \emph{admissible-type} (respectively, of \emph{decusp-type}; of \emph{geo-decusp-type}) if $f$ is of cov-type (respectively, if $f$ is an isomorphism; if $f$ is of geo-iso-type).
\end{defi}

\begin{prop}\label{well_defined_for_kernel_cusp}
    In the situation of \cref{defi_CAVC_morphism_type}(v), the subset $S$ is uniquely determined by the group homomorphism $f:G_{\Delta}\to G_{\Delta}'$. We refer to $S$ as the kernel class of $f$.
\end{prop}
\begin{proof}
    By [\cite{SZM}, Theorem 4.20(ii)(iii)], one may reduce the claim to the corresponding assertion for hyperbolic curves. Moreover, after passing to a suitable open subgroup, one may assume that the curve that corresponds to $G_{\Delta}$ is of type $(g,r)$ with $r\geq 3$. Then the assertion follows immediately from [\cite{MZK3}, Proposition A.8(ii)(iii)].
\end{proof}

The following corollary follows immediately from \cref{defi_CAVC_morphism_type}.

\begin{coro}\label{abs_nonsense_CAVC_morphism}
    Let $G_{\Delta}\subset G_{\Pi}$, $G_{\Delta}'\subset G_{\Pi}'$ be inclusions of gen-CAVC-type; $f:G_{\Pi}\to G_{\Pi}'$ a morphism of inclusions. Then:

    (i) $f$ is an isomorphism if and only if $f$ is of geo-iso-type and decusp-type.

    (ii) $f$ is of geo-iso-type if and only if $f$ is of cov-type and geo-decusp-type.

    (iii) If $f$ is of decusp-type, then $f$ is of geo-decusp-type.

    (iv) If $f$ is either of cov-type or of geo-decusp-type, then $f$ is of admissible-type.

    (v) Assume further that $G_{\Delta}''\subset G_{\Pi}''$ is an inclusion of gen-CAVC-type; $f':G'_{\Pi}\to G_{\Pi}''$ a morphism of inclusions. Then if both $f$ and $f'$ are of geo-iso-type (respectively, of cov-type; of decusp-type; of geo-decusp-type), then their composition $f'\circ f$ is again of geo-iso-type (respectively, of cov-type; of decusp-type; of geo-decusp-type).

    (vi) $f$ is of geo-decusp-type if and only if there exist an inclusion $G_{\Delta}''\subset G_{\Pi}''$ of gen-CAVC-type and two morphisms $g:G_{\Pi}''\to G_{\Pi}'$, $h:G_{\Pi}\to G_{\Pi}''$ of inclusions such that $f = g\circ h$, $g$ is of geo-iso-type, and $h$ is of decusp-type.

    (vii) $f$ is of admissible-type if and only if there exist an inclusion $G_{\Delta}''\subset G_{\Pi}''$ of gen-CAVC-type and two morphisms $g:G_{\Pi}''\to G_{\Pi}'$, $h:G_{\Pi}\to G_{\Pi}''$ of inclusions such that $f = g\circ h$, $g$ is of cov-type, and $h$ is of decusp-type.
\end{coro}


\begin{lemm}\label{decusp_universal} 
    Let $G_{\Delta}\subset G_{\Pi}$ be an inclusion of CAVC-type; $G_{\Delta}'\subset G_{\Pi}'$ an inclusion of gen-CAVC-type; $f:G_{\Pi}\to G_{\Pi}'$ a morphism of inclusions of geo-decusp-type. Then for each representative $I\subset G_{\Delta}$ of an element $\alpha\in \cusp(G_{\Pi})$, one of the following holds:
    \begin{enumerate}
        \item $f(I) = \{e\}$;
        \item $f$ induces an isomorphism $I \simto J\subset G_{\Delta}'$, where $J$ is a subgroup of $G_{\Delta}'$ that represents an element of $\cusp(G_{\Pi}')$.
    \end{enumerate}
    If, moreover, $f$ is of decusp-type, and both $I^{\dagger}$ and $I^{\ddagger}$ are representatives of two \emph{distinct} elements of $\cusp(G_{\Pi})$ and satisfy (2), then their respective images in $G_{\Delta}'$ represent two \emph{distinct} elements of $\cusp(G_{\Pi}')$.
\end{lemm}
\begin{proof}
    First, we assume that $f$ is of decusp-type. By \cref{abs_nonsense_CAVC_morphism}(iii)(iv), $f$ is of admissible-type.
    Let $S$ be the kernel class of $f$ (cf. \cref{well_defined_for_kernel_cusp}).
    By \cref{defi_CAVC_morphism_type}(iii) and \cref{decuspidalization_curve_theoretic}, there exists a VD $(\mathcal{C},\phi)$ of $G_{\Delta}\subset G_{\Pi}$ together with a decuspidalization $h:\mathcal{C}\to\mathcal{C}'$ that determines the natural image of $S$ in $\cusp(\mathcal{C})$. Let $h_{\ast}:\Pi_{\mathcal{C}}\to\Pi_{\mathcal{C}'}$ be the morphism of virtual fundamental groups induced by $h$ (cf. [\cite{SZM}, Corollary 3.4(i)]). One sees immediately that there exists a unique isomorphism $\phi':\Pi_{\mathcal{C}'}\simto G_{\Pi}'$ such that $f\circ \phi = \phi'\circ h_{\ast}$. Moreover, $(\mathcal{C}',\phi')$ is a VD (respectively, cls-VD) if $G_{\Delta}'\subset G_{\Pi}'$ is of CAVC-type (respectively, of cls-CAVC-type). Therefore, the assertion follows immediately from [\cite{SZM}, Proposition 5.3(ii)]. The general case follows immediately from \cref{abs_nonsense_CAVC_morphism}(vi) and [\cite{SZM}, Theorem 4.20(ii)].
\end{proof}

\begin{defi}\label{determined_cusp_morphism}\quad

    (i) Let $G_{\Delta}\subset G_{\Pi}$, $G_{\Delta}'\subset G_{\Pi}'$ be inclusions of gen-CAVC-type; $f:G_{\Pi}\to G_{\Pi}'$ a morphism of inclusions of cov-type. By \cref{covering_group_theoretic}, $f$ induces a natural surjection $h:\cusp(G_{\Pi})\surjto \cusp(G_{\Pi}')$. Then we say that $f$ \emph{determines} $h$.
    
    (ii) Let $G_{\Delta}\subset G_{\Pi}$ be an inclusion of CAVC-type; $G_{\Delta}'\subset G_{\Pi}'$ an inclusion of gen-CAVC-type; $f:G_{\Pi}\to G_{\Pi}'$ a morphism of inclusions of decusp-type. By \cref{decusp_universal}, $f$ induces a natural injection $h:\cusp(G_{\Pi}')\injto \cusp(G_{\Pi})$. Then we say that $f$ \emph{determines} $h$. 
\end{defi}

\begin{remm} 
    In \cref{determined_cusp_morphism}, if both assumptions that are discussed in the situations (i) and (ii) are satisfied, then $f$ determines two maps between the sets of cusps in opposite directions. However, in this case, the morphisms of inclusions are isomorphisms (cf. \cref{abs_nonsense_CAVC_morphism}(i)(ii)(iii)), so the two sets of cusps are naturally isomorphic. In particular, there is no ambiguity. 
\end{remm}

\begin{defi}\label{defi_constructible_inclusion}\quad

    (i) Let $\Delta\subset \Pi$ be an inclusion of gen-CAVC-type. A collection of \emph{geometric data}, abbreviated as \emph{GD}, of $\Delta\subset \Pi$ is defined to be a tuple $(X,\phi)$, where $X$ is a hyperbolic curve over an AF (cf. \cref{Notations and Terminologies}, the discussion entitled ``Fields, Schemes and Curves'') and $\phi:\Pi_{X}\simto \Pi$ is a group isomorphism satisfying $\phi(\Delta_{X})=\Delta$.

    (ii) Let $\Delta\subset \Pi$ be an inclusion of gen-CAVC-type. Then we say that $\Delta\subset \Pi$ is \emph{constructible} if there exists a GD of $\Delta\subset \Pi$.
\end{defi}

\begin{coro}\label{constructible_for_geometric_section}
    Let $\Delta\subset \Pi$ be an inclusion of gen-CAVC-type; $((X,Y,f,k,\bar{k},t,s), \phi)$ a VD of $\Delta\subset \Pi$ such that $k$ is an AF. Assume that $s$ is induced by a rational point of $Y$. Then $\Delta\subset \Pi$ is constructible.
\end{coro}
\begin{proof}
    The assertion follows immediately from [\cite{SZM}, Corollary 3.8].
\end{proof}

\begin{prop}\label{anabelian_CAVC}
    Let $\Delta\subset \Pi$ be a constructible inclusion of gen-CAVC-type. Then the following assertions hold:

    (i) Let $(X^\dagger,\phi^\dagger)$, $(X^\ddagger,\phi^\ddagger)$ be two GDs of $\Delta\subset \Pi$. Then there exists a \emph{unique} isomorphism $f: X^\dagger\to X^\ddagger$ of schemes such that $f_{\ast} = (\phi^\ddagger)^{-1}\circ \phi^\dagger$. Equivalently, a GD is unique up to canonical isomorphism.

    (ii) One may construct the following data:
    \begin{enumerate}
        \item an ACAF (cf. \cref{Notations and Terminologies}, the discussion entitled ``Fields, Schemes and Curves''), which we denote by $\base(\Delta\subset \Pi)$;
        \item a hyperbolic curve over $\base(\Delta\subset \Pi)$, which we denote by $\sch(\Delta\subset \Pi)$;
        \item an FFUC (cf. \cref{Notations and Terminologies}, the discussion entitled ``Fields, Schemes and Curves'') of $X$, which we denote by $\tilde{K}(\Delta\subset \Pi)$;
        \item a natural inclusion $\base(\Delta\subset \Pi)\injto K(\sch(\Delta\subset \Pi))$ (cf. \cref{Notations and Terminologies}, the discussion entitled ``Fields, Schemes and Curves'') of fields, which we denote by $\alpha(\Delta\subset \Pi)$;
        \item a natural inclusion $K(\sch(\Delta\subset \Pi))\injto \tilde{K}(\Delta\subset \Pi)$ of fields, which we denote by $\beta(\Delta\subset \Pi)$,
    \end{enumerate}
    from $\Delta\subset \Pi$ functorially with respect to isomorphisms of topological groups such that the following properties hold:
    \begin{enumerate}
        \item The action of $\Pi$ on $\base(\Delta\subset \Pi)$ functorially induced by the conjugation action of $\Pi$ on itself factors through the quotient $\Pi\surjto \Pi/\Delta$. Moreover, this action induces a natural isomorphism $\Pi/\Delta\simto \gal (\base(\Delta\subset \Pi)/(\base(\Delta\subset \Pi))^{\Pi/\Delta})$.
        \item The action of $\Pi$ on $\sch(\Delta\subset \Pi)$ functorially induced by the conjugation action of $\Pi$ on itself factors through the quotient $\Pi\surjto \Pi/\Delta$. Moreover, $\base(\Delta\subset \Pi)$ admits an embedding into $\tilde{K}(\Delta\subset \Pi)$ via $\alpha(\Delta\subset \Pi)$ and $\beta(\Delta\subset \Pi)$. Furthermore, the quotient $(\sch(\Delta\subset \Pi))^{\Pi/\Delta}$ of $\sch(\Delta\subset \Pi)$ by the action of $\Pi/\Delta$ exists and is a hyperbolic curve over $(\base(\Delta\subset \Pi))^{\Pi/\Delta}$.
        \item The action of $\Pi$ on $\tilde{K}(\Delta\subset \Pi)$ functorially induced by the conjugation action of $\Pi$ on itself, together with the natural inclusion $\beta(\Delta\subset \Pi)$, induces a natural isomorphism $\phi(\Delta\subset \Pi):\Pi_{(\sch(\Delta\subset \Pi))^{\Pi/\Delta}}\simto \Pi$ up to inner automorphism.
    \end{enumerate}
    Moreover, $((\sch(\Delta\subset \Pi))^{\Pi/\Delta}, \phi(\Delta\subset \Pi))$ is a GD of $\Delta\subset \Pi$. We shall say that $((\sch(\Delta\subset \Pi))^{\Pi/\Delta}, \phi(\Delta\subset \Pi))$ is the \emph{canonical GD} associated to $\Delta\subset \Pi$, and that $(\sch(\Delta\subset \Pi))^{\Pi/\Delta}$ is the \emph{canonical scheme} associated to $\Delta\subset \Pi$.
\end{prop}
\begin{proof}
    Assertion (i) follows immediately from [\cite{RNSPM}, Theorem D; \cite{MZK6}, Corollary 1.3.5]. 
    Assertion (ii) follows immediately from [\cite{RNSPM}, Theorem D; \cite{MZK6}, Corollary 1.3.5], together with  
    the functoriality with respect to finite \'etale coverings of the functorial algorithms implicit in the proofs of these results.
\end{proof}

\begin{coro}\label{anabelian_CAVC_more}
    Let $G_{\Delta}\subset G_{\Pi}$, $G_{\Delta}'\subset G_{\Pi}'$ be inclusions of gen-CAVC-type such that there exists a GD $(X,\phi)$ of $G_{\Delta}\subset G_{\Pi}$. Then the following assertions hold:

    (i) Suppose that there exists a morphism $f:G_{\Pi}'\to G_{\Pi}$ of cov-type. Then $G_{\Delta}'\subset G_{\Pi}'$ is constructible.

    (ii) Suppose that there exists a morphism $f:G_{\Pi}\to G_{\Pi}'$ of decusp-type. Then $G_{\Delta}'\subset G_{\Pi}'$ is constructible.

    (iii) Suppose that there exists a morphism $f:G_{\Pi}\to G_{\Pi}'$ of geo-iso-type. Then $G_{\Delta}'\subset G_{\Pi}'$ is constructible.

    (iv) In the situation of (i), let $(X',\phi')$ be a GD of $G_{\Delta}'\subset G_{\Pi}'$. Then there exists a unique morphism of schemes $h:X'\to X$ such that $h_{\ast} = \phi^{-1}\circ f\circ \phi'$ (up to inner automorphism). We shall say that $h$ is the morphism (of GDs) induced by $f$.

    (v) In the situation of (ii), let $(X',\phi')$ be a GD of $G_{\Delta}'\subset G_{\Pi}'$. Then there exists a unique morphism of schemes $h:X\to X'$ such that $h_{\ast} = (\phi')^{-1}\circ f\circ \phi$ (up to inner automorphism). We shall say that $h$ is the morphism (of GDs) induced by $f$.

    (vi) In the situation of (iv), the morphism of canonical GDs induced by $f$ can be group-theoretically reconstructed from $f$.

    (vii) In the situation of (v), the morphism of canonical GDs induced by $f$ can be group-theoretically reconstructed from $f$.
\end{coro}
\begin{proof}
    First, we consider assertion (i). Let $H\coloneqq\phi^{-1}(f(G_{\Pi}'))\subset \Pi_{X}$ be the subgroup of $\Pi_{X}$ corresponding to $f(G_{\Pi}')$. Since $H$ is an open subgroup, there exists an \etale covering $h:Y\to X$ arising from $H$. Let $\phi':\Pi_{Y}\simto G_{\Pi}'$ be a group isomorphism such that $f\circ \phi' = \phi \circ h_{\ast}$. Then one verifies immediately that $(Y, \phi')$ is a GD of $G_{\Delta}'\subset G_{\Pi}'$. 
    Next, we consider assertion (ii). By [\cite{SZM}, Theorem 4.20(iii)], there exists a decuspidalization $h:X\to Y$ such that the kernel of the induced morphism $h_{\ast}$ of fundamental groups is topologically generated by the image of the kernel class of $f$ (cf. \cref{well_defined_for_kernel_cusp}) under $\phi^{-1}$. Let $\phi':\Pi_{Y}\simto G_{\Pi}'$ be a group isomorphism such that $\phi'\circ h_{\ast}=f\circ\phi$.  Then 
    one verifies immediately that $(Y, \phi')$ is a GD of $G_{\Delta}'\subset G_{\Pi}'$.
    Next, we consider assertion (iii). By assertion (i), we may assume, without loss of generality, that the image of $G_{\Pi}$ in $G_{\Pi}'$ under $f$ is normal. Let $I$ be the quotient $G_{\Pi}/G_{\Delta}$; $J$ the quotient $G_{\Pi'}/G_{\Delta'}$. By definition, $f$ induces an injection $f':I\injto J$. By \cref{anabelian_CAVC}(i), it follows that $H\coloneqq J/I$ acts on $X$ via $\phi^{-1}$. Let $C\subset J$ be the centralizer of $I$ in $J$. Since $I$ is slim (cf. [\cite{MZK3}, Theorem 1.7(ii)(iii)]), it follows that $C\cap I = \{ 1 \}$, hence that $C$ is finite. Since $C$ is normal, and $J$ is elastic (cf. [\cite{MZK3}, Theorem 1.7(ii)(iii)]), we thus conclude that $C= \{ 1 \}$. Let $k$ be the base field of $X$. Then since $C=\{ 1 \}$, it follows immediately that the induced action of $H$ on $k$ is faithful. Let $X^H$ be the scheme-theoretic quotient (cf. \cref{Notations and Terminologies}, the discussion entitled ``Fields, Schemes and Curves'', which is applicable in 
    light of well-known elementary properties of algebraic curves). Since the action of $H$ on $k$ is faithful, the natural morphism $h: X\to X^H$ is a Galois \etale covering. Moreover, by considering the natural injections $\Pi_{X^H}\injto \Aut(\Pi_{X})$, $G_{\Pi}'\injto \Aut(G_{\Pi})$ induced by the conjugation actions, one sees that there exists an isomorphism $\phi':\Pi_{X^H}\to G_{\Pi}'$ such that $f\circ\phi = \phi'\circ h_{\ast}$. In particular, one verifies immediately that $(X^H, \phi')$ is a GD of $G_{\Delta}'\subset G_{\Pi}'$.

    Assertions (iv) and (vi) follow immediately from \cref{anabelian_CAVC}(i)(ii) and the proof of assertion (i). Assertions (v) and (vii) follow immediately from \cref{anabelian_CAVC}(i)(ii) and the proof of assertion (ii).
\end{proof}

\begin{defi}\label{defi_trivial_family}\quad

    (i) Let $f:X \to Y$ be a family of hyperbolic curves such that $Y$ is a normal scheme of finite type over a field $k$ of characteristic zero. Then we say that $f$ is a \emph{split family of curves} if there exists a morphism $g:X \to Z$, where $Z$ is a smooth hyperbolic curve over $k$, such that the natural morphism
    \begin{displaymath}
        (f,g): X \longrightarrow Y \times_{k} Z
    \end{displaymath}
    is an isomorphism. In this situation, we shall say that $g$ is a \emph{splitting morphism} of $f$. 
    
    (ii) If $\mathcal{C} = (X,Y,f,k,\bar{k},t,s)$ is a pointed virtual curve, then we say that $\mathcal{C}$ is a 
    \emph{split pointed virtual curve} if $f$ is a split family of curves.
\end{defi}
\begin{coro}\label{constructible_trivial_family}
    Let $\mathcal{C} = (X,Y,f,k,\bar{k},t,s)$ be a split pointed virtual curve over an AF (cf. \cref{Notations and Terminologies}, the discussion entitled ``Fields, Schemes and Curves''; [\cite{SZM}, Proposition 5.2]) $k$; $g:X \to Z$ a splitting morphism of $f$. Then the inclusion
    \begin{displaymath}
        \Delta_{\mathcal{C}} \subset \Pi_{\mathcal{C}}
    \end{displaymath}
    is naturally isomorphic to the inclusion
    \begin{displaymath}
        \Delta_{Z} \subset \Pi_{Z}.
    \end{displaymath}
    Moreover, this natural isomorphism induces a natural bijection between $\cusp(\Pi_{\mathcal{C}})$ and $\cusp(Z)$ (cf. [\cite{SZM}, Definition 4.19(ii)]), and the inclusion $\Delta_{\mathcal{C}} \subset \Pi_{\mathcal{C}}$ is constructible.
\end{coro}
\begin{proof}
    The assertion follows immediately from \cref{defi_trivial_family}(i)(ii) and [\cite{SZM}, Lemma 5.11(i)(ii)].
\end{proof}

\begin{prop}\label{virtual_inertia_and_real_inertia}
    Let $\mathcal{C} = (X,Y,f,k,\bar{k},t,s)$ be a pointed virtual curve 
    such that $k$ is of characteristic zero. Suppose that $k$ is $l$-cyclotomically full for some prime number $l$.
    Write $\cusp(f)\coloneqq \cusp(f,\Pi_{X})$ (cf. \cref{Notations and Terminologies}, the discussion entitled ``Fundamental Groups''). Note that the kernel of $f_{\ast}:\Pi_{X}\to \Pi_{Y}$, regarded as a closed subgroup of $\Pi_{X}$, is equal to $\Delta_{\mathcal{C}}$. Therefore, one may regard $\cusp(f)$ as a set of conjugacy classes of subgroups of $\Delta_{\mathcal{C}}$. Then $\cusp(f) = \cusp(\mathcal{C})$.
\end{prop}
\begin{proof}
    By [\cite{SZM}, Theorem 4.20(ii)], without loss of generality, one may assume that $s$ arises from a rational point of $Y$. The assertion then follows immediately from [\cite{SZM}, Theorem 4.20(iii)].
\end{proof}


\section{Virtual Decuspidaloids}\label{section_virtual_decuspidaloid}

In this section, we introduce the notion of a \emph{virtual decuspidaloid} and study its properties, as well as various associated anabelian constructions.

\begin{defi}
    Denote by $\ig$ the category whose objects are inclusions $H\subset G$ of profinite groups, and whose morphisms are morphisms of inclusions (cf. \cref{Notations and Terminologies}, the discussion entitled ``Groups and Topologies'').
\end{defi}

\begin{lemm}\label{transitive_cusp}
    Let $G_{\Delta}\subset G_{\Pi}$ be an inclusion of CAVC-type. Let $H$, $H'$ be two open subgroups of $G_{\Pi}$ such that $H'\subset H$. Denote by $S$ (respectively, $S'$) the set $\cusp(H\cap G_{\Delta}\subset H)$ (respectively, $\cusp(H'\cap G_{\Delta}\subset H')$) (cf. \cref{defi_cusp_33}; \cref{covering_group_theoretic}) and by $f: S\to\cusp(G_{\Delta}\subset G_{\Pi})$ (respectively, $f': S'\to\cusp(G_{\Delta}\subset G_{\Pi})$, $f'':S'\to S$) the induced morphism of sets of conjugacy classes of $G_{\Delta}$, $H$ (respectively, $G_{\Delta}$, $H'$; $H$, $H'$) (cf. \cref{defi_cusp_compatibility_subgroup}). Then $f'=f\circ f''$.
\end{lemm}
\begin{proof}
    \cref{transitive_cusp} follows immediately from the construction in \cref{defi_cusp_compatibility_subgroup}.
\end{proof}

\begin{defi}\label{defi_virtual_decuspidaloid}
    Let $G_{\Delta}\subset G_{\Pi}$ be an inclusion of CAVC-type. Then we define a category $\cavc_{G_{\Delta}\subset G_{\Pi}}$ in the following way:

    (i) The objects are tuples $(H,S)$, where $H$ is an open subgroup of $G_{\Pi}$ and $S$ is a subset of admissible-type of $\cusp(H\cap G_{\Delta}\subset H)$ (cf. \cref{defi_admissible_type}; \cref{covering_group_theoretic}).

    (ii) For two objects $(H,S)$ and $(H',S')$, define the set of morphisms $\hom((H,S),(H',S'))$ as follows:
    \begin{enumerate}
        \item a singleton set, if $H\subset H'$, and $S$ is contained in the set of conjugacy classes induced by $S'$ under the inclusion $H\cap G_{\Delta} \subset H'\cap G_{\Delta}$ (cf. \cref{defi_cusp_compatibility_subgroup});
        \item the empty set, otherwise.
    \end{enumerate}

    (iii) Let $(H,S)$, $(H',S')$, and $(H'',S'')$ be three objects. Assume further that there exist morphisms $f:(H,S)\to(H',S')$, $f':(H',S')\to (H'',S'')$. By \cref{transitive_cusp}, there exists a morphism $f''$ from $(H,S)$ to $(H'',S'')$. As there is at most one morphism between any two objects, one defines $f''\coloneqq f'\circ f$.
\end{defi}

\begin{defi}\label{CAVC_functor_defi}
    Let $G_{\Delta}\subset G_{\Pi}$ be an inclusion of CAVC-type. Then we define a functor
    \begin{displaymath}
        F_{G_{\Delta}\subset G_{\Pi}}: \cavc_{G_{\Delta}\subset G_{\Pi}}\to \ig
    \end{displaymath}
    by the following steps:
    \begin{enumerate}
        \item For an object $(H,S)\in \obj(\cavc_{G_{\Delta}\subset G_{\Pi}})$, associate to it the inclusion $((H\cap G_{\Delta})/S)\subset H/S$ (cf. \cref{Notations and Terminologies}, the discussion entitled ``Groups and Topologies'').

        \item For each morphism $(H,S)\to (H',S')$ of $\cavc_{G_{\Delta}\subset G_{\Pi}}$, associate to it the natural morphism from $H/S$ to $H'/S'$, regarded as subquotients of $G_{\Pi}$.  [One verifies easily that this morphism of inclusions is 
        well-defined.]
    \end{enumerate}
    We refer to the pair $(\cavc_{G_{\Delta}\subset G_{\Pi}},F_{G_{\Delta}\subset G_{\Pi}})$ as the \emph{virtual decuspidaloid} associated to $G_{\Delta}\subset G_{\Pi}$. Moreover, if $\mathcal{C}$ is a pointed virtual curve such that $\Delta_{\mathcal{C}}\subset \Pi_{\mathcal{C}}$ is of CAVC-type, then we also refer to the data $(\cavc_{\Delta_{\mathcal{C}}\subset \Pi_{\mathcal{C}}},F_{\Delta_{\mathcal{C}}\subset \Pi_{\mathcal{C}}})$ as the \emph{virtual decuspidaloid} associated to $\mathcal{C}$.
\end{defi}

\begin{theo}\label{construction_CAVC}
    Let $G_{\Delta}\subset G_{\Pi}$ be an inclusion of CAVC-type. Then the functor $F_{G_{\Delta}\subset G_{\Pi}}$ (cf. \cref{CAVC_functor_defi}) can be reconstructed group-theoretically and functorially from the abstract inclusion $G_{\Delta}\subset G_{\Pi}$.
\end{theo}
\begin{proof}
    \cref{construction_CAVC} follows immediately from \cref{defi_virtual_decuspidaloid}, \cref{CAVC_functor_defi}, and [\cite{SZM}, Theorem 4.20(ii)].
\end{proof}

\begin{lemm}\label{open_unique}
    In the situation of \cref{CAVC_functor_defi}, suppose that the naturally induced morphism from $H/S$ to $H'/S'$ discussed in procedure $(2)$ is injective. Then every representative $I$ of an element of $S'$ is a subgroup of $H$.
\end{lemm}
\begin{proof}
    By [\cite{SZM}, Corollary 3.8; \cite{SZM}, Theorem 4.20(ii)], one may reduce the claim to the corresponding assertion for hyperbolic curves. The conclusion then follows immediately from well-known properties of the corresponding \etale fundamental groups, by considering the issue of ramification at the cusps in light of the Hurwitz formula.
\end{proof}

\begin{defi}\label{defi_gen_virtual_decuspidaloid}
    We say that an $\ig$-structured category $(\mf{A},A)$ is a \emph{generalized virtual decuspidaloid} if there exists an isomorphism 
    \begin{displaymath}
        (F,\tau):(\mf{A},A)\to (\mf{B},B)
    \end{displaymath}
    of $\ig$-structured categories, where $(\mf{B},B)$ is the virtual decuspidaloid associated to an inclusion of CAVC-type.
\end{defi}

\begin{defi}\label{defi_virtual_decuspidaloid_revert_to_group}
    Let $(\mathfrak{X},F)$ be a generalized virtual decuspidaloid. Let $O$ be an object of $\mathfrak{X}$. Then the image $F(O)$ is an object of $\ig$. We shall write $\Pi(O)$ (respectively, $\Delta(O)$) for the codomain (respectively, domain) of the inclusion determined by $F(O)$. 
    Since $(\mathfrak{X},F)$ is a generalized virtual decuspidaloid, it follows that the inclusion $\Delta(O)\subset \Pi(O)$ is of gen-CAVC-type.
    We shall write $G(O)\coloneqq\Pi(O)/\Delta(O)$; $\cusp(O)\coloneqq\cusp(F(O))$ (cf. \cref{remm_37}). 
    Let $f:O\to O'$ be a morphism of $\mathfrak{X}$. Denote by $f_{\ast}$ the naturally induced morphism from $\Pi(O)$ to $\Pi(O')$; by $f_{\ast}^\Delta$ the naturally induced morphism from $\Delta(O)$ to $\Delta(O')$; by $f_{\ast}^{G}$ the naturally induced morphism from $G(O)$ to $G(O')$.
\end{defi}

\begin{defi}\label{defi_virtual_decuspidaloid_over_nf_mlf}
    Let $(\mathfrak{X},F)$ be an $\ig$-structured category. Then we say that $(\mathfrak{X},F)$ is of \emph{global-type} (respectively, \emph{local-type}; \emph{arithmetic-type}) if there exists an inclusion $\Delta\subset \Pi$ of profinite groups arising from a pointed virtual curve defined over a number field (respectively, over a mixed-characteristic local field; over either a number field or a mixed-characteristic local field), together with an isomorphism
    \begin{displaymath}
    (\mathfrak{X},F)\simto(\cavc_{\Delta\subset \Pi},F_{\Delta\subset \Pi})
    \end{displaymath}
    of $\ig$-structured categories.
\end{defi}

\begin{defi}\label{defi_intrinsic_virtual_decuspidaloid}
    Let $(\mathfrak{X},F)$ be a generalized virtual decuspidaloid. 

    (i) Let $f:O\to O'$ be a morphism of $\mathfrak{X}$. Then we say that $f$ is of \emph{geo-iso-type} (respectively, of \emph{cov-type}; of \emph{decusp-type}; of \emph{geo-decusp-type}, \emph{admissible-type}) if the induced morphism $f_{\ast}$ is of geo-iso-type (respectively, of cov-type; of decusp-type; of geo-decusp-type; of admissible-type) (cf. \cref{defi_CAVC_morphism_type}).

    (ii) Let $O$ be an object of $\mf{X}$. Then we say that $O$ is a \emph{root} of $\mf{X}$ if every morphism in $\hom(-,O)$ is of cov-type and every morphism in $\hom(O,-)$ is of decusp-type.

    (iii) Let $O,O'$ be objects of $\mf{X}$. Then we say that $O'$ is a \emph{field base change} (respectively, \emph{decuspidalization}; \emph{covering}) of $O$ if there exists a morphism 
    \begin{displaymath}
    f: O'\to O\ (\text{respectively, }f:O\to O',\ f:O'\to O)
    \end{displaymath}
    such that $f$ is of geo-iso-type (respectively, of decusp-type; of cov-type). We remark that such a morphism is unique if it exists (cf. \cref{defi_virtual_decuspidaloid}(ii)).

    (iv) Let $O,O'$ be objects of $\mf{X}$. Then we say that $O'$ is a \emph{subjugate} of $O$ if there exists an object $O''\in\mf{X}$ such that $O'$ is a decuspidalization of $O''$ and $O''$ is a covering of $O$.

    (v) Let $O$ be an object of $\mf{X}$. Then we say that $O$ is \emph{proper} (respectively, \emph{non-proper}) if the inclusion $\Delta(O)\subset\Pi(O)$ is of cls-CAVC-type (respectively, of CAVC-type).

    (vi) Let $O$ be an object of $\mf{X}$. Then we say that $O$ is of type $(g,r)$ (respectively, genus $g$) if the inclusion $\Delta(O)\subset\Pi(O)$ is of type $(g,r)$ (respectively, genus $g$).
\end{defi}

\begin{remm}
    By [\cite{SZM}, Proposition 5.7], the subjugate relation defined in \cref{defi_intrinsic_virtual_decuspidaloid}(iv) is transitive.
\end{remm}

\begin{coro}\label{prop_intrinsic_virtual_decuspidaloid}
    Let $(\mathfrak{X},F)$ be a generalized virtual decuspidaloid.

    (i) Let $O$, $O'$ be two objects of $\mf{X}$ that are isomorphic. Then $O=O'$.

    (ii) There exists a \emph{unique} root of $\mf{X}$. Denote this object by $\root(\mf{X})$.

    (iii) Let $O$ be an object of $\mf{X}$. Then $O$ is a subjugate of $\root(\mf{X})$.

    (iv) Let $O$ be an object of $\mf{X}$; $H\subset \Pi(O)$ an open subgroup. Then there exists a unique covering $O'$ of $O$ such that the image of $\Pi(O')$ via the unique morphism from $O'$ to $O$ coincides with $H$. We say that $O'$ is the covering of $O$ induced by $H$.

    (v) Let $O$, $O'$ be two objects of $\mf{X}$ such that $O'$ is a decuspidalization of $O$; $f'$ the unique morphism from $O$ to $O'$. Then there exists a unique $G(O)$-invariant subset $S\subset \cusp(O)$ such that the kernel of $f_{\ast}$ is the closed normal subgroup of $\Delta(O)$ generated by the union of all representatives of the conjugacy classes in $S$ (cf. \cref{well_defined_for_kernel_cusp}). We say that $S$ is the \emph{kernel class} of $\cusp(O)$ determined by the morphism $f': O \to O'$ of decusp-type. 
    Moreover, suppose that $O''\in\obj(\mf{X})$ is another decuspidalization of $O$ such that $S$ coincides with the kernel class of $\cusp(O)$ determined by the morphism $f'': O \to O''$ of decusp-type. Then $O''=O'$. Hence $O'$ is uniquely determined by $S$.
    We shall say that $O'$ is the decuspidalization of $O$ induced by $S$.
\end{coro}
\begin{proof}
    Assertions (i), (ii) and (iii) follow immediately from 
    \cref{defi_virtual_decuspidaloid}, \cref{CAVC_functor_defi}, \cref{defi_gen_virtual_decuspidaloid}, and \cref{defi_intrinsic_virtual_decuspidaloid}. Assertion (iv) follows immediately from \cref{open_unique}. Assertion (v) follows immediately from \cref{well_defined_for_kernel_cusp} and \cref{open_unique}.
\end{proof}

\begin{remm}
    Let $(\mathfrak{X},F)$ be a generalized virtual decuspidaloid. For any object $O$ of $\mf{X}$, $\Pi(O)$ (respectively, $\Delta(O)$) is a subquotient of $\Pi(\root(\mf{X}))$. Hence one may regard $\Pi(O)$ (respectively, $\Delta(O)$) as a subquotient of $\Pi(\root(\mf{X}))$. For instance, for two objects $O$, $O'$ of $\mf{X}$, it makes sense to say that $\Delta(O)=\Delta(O')$. Moreover, the notion of a virtual decuspidaloid serves as a convenient terminological apparatus for referring to the various subquotients of a virtual fundamental group of CAVC-type that naturally arise from geometric objects.
\end{remm}

\begin{defi}\label{defi_constructible_object}
    Let $(\mathfrak{X},F)$ be a generalized virtual decuspidaloid; $O$ an object of $\mf{X}$. Then we shall say that $O$ is \emph{constructible} if $F(O)$ is constructible (cf. \cref{defi_constructible_inclusion}(ii)).
\end{defi}

\begin{prop}\label{constructible_via_subjugate}
    Let $(\mathfrak{X},F)$ be a generalized virtual decuspidaloid; $O$, $O'$ objects of $\mf{X}$ such that $O'$ is a subjugate of $O$. Assume further that $O$ is constructible. Then $O'$ is constructible.
\end{prop}
\begin{proof}
    The assertion follows immediately from \cref{anabelian_CAVC_more}(i)(ii) and \cref{defi_intrinsic_virtual_decuspidaloid}(iv).
\end{proof}

\begin{defi}\label{defi_virtual_decuspidaloid_arises_from_object}
    Let $(\mathfrak{X},F)$ be a generalized virtual decuspidaloid; $O$ an object of $\mf{X}$. 
    
    (i) Denote by $\mathfrak{X}|_{O}$ the full subcategory of $\mathfrak{X}$ determined by the set of all subjugates of $O$; by $I_{O}:\mathfrak{X}|_{O}\injto \mathfrak{X}$ the naturally induced inclusion functor; by $F|_{O}$ the restriction of $F$ to $\mathfrak{X}|_{O}$. Write $\id_{F|_{O}}$ for the identity natural transformation on $F|_{O}$.  Then it follows immediately that $(I_{O},\id_{F|_{O}}):(\mathfrak{X}|_{O},F|_{O})\injto(\mathfrak{X},F)$ defines an immersion of $\ig$-structured categories (cf. \cref{Notations and Terminologies}, the discussion entitled ``Graphs and Categories''). We shall refer to $\mathfrak{X}|_{O}$ as the \emph{subjugate subcategory} of $\mathfrak{X}$ induced by $O$ and to $(I_{O},\id_{F|_{O}})$ as the natural \emph{full-immersion} of $\ig$-structured categories induced by $O$.

    (ii) Denote by $\mathfrak{X}|_{O}^{\cov}$ the full subcategory of $\mathfrak{X}$ determined by the set of all coverings of $O$; by $I_{O}^{\cov}:\mathfrak{X}|_{O}^{\cov}\injto \mathfrak{X}$ the naturally induced inclusion functor; by $F|_{O}^{\cov}$ the restriction of $F$ to $\mathfrak{X}|_{O}^{\cov}$. Write $\id_{F|_{O}^{\cov}}$ for the identity natural transformation on $F|_{O}^{\cov}$. Then it follows immediately that $(I_{O}^{\cov},\id_{F|_{O}^{\cov}}):(\mathfrak{X}|_{O}^{\cov},F|_{O}^{\cov})\injto(\mathfrak{X},F)$ defines an immersion of $\ig$-structured categories. We shall refer to $\mathfrak{X}|_{O}^{\cov}$ as the \emph{covering subcategory} of $\mathfrak{X}$ induced by $O$ and to $(I_{O}^{\cov},\id_{F|_{O}^{\cov})}$ as the natural \emph{covering-immersion} of $\ig$-structured categories induced by $O$.

    (iii) Denote by $\mathfrak{X}|_{O}^{\bc}$ the full subcategory of $\mathfrak{X}$ determined by the set of all field base changes of $O$; by $I_{O}^{\bc}:\mathfrak{X}|_{O}^{\bc}\injto \mathfrak{X}$ the naturally induced inclusion functor; by $F|_{O}^{\bc}$ the restriction of $F$ to $\mathfrak{X}|_{O}^{\bc}$. Write $\id_{F|_{O}^{\bc}}$ for the identity natural transformation on $F|_{O}$. Then it follows immediately that $(I_{O}^{\bc},\id_{F|_{O}^{\bc}}):(\mathfrak{X}|_{O}^{\bc},F|_{O}^{\bc})\injto(\mathfrak{X},F)$ defines an immersion of $\ig$-structured categories. We shall refer to $\mathfrak{X}|_{O}^{\bc}$ as the \emph{base change subcategory} of $\mathfrak{X}$ induced by $O$ and to $(I_{O}^{\bc},\id_{F|_{O}^{\bc}})$ as the natural \emph{base-change-immersion} of $\ig$-structured categories induced by $O$.
\end{defi}

It follows immediately from \cref{defi_virtual_decuspidaloid_arises_from_object} that the following result holds.

\begin{coro}
    Let $(\mathfrak{X},F)$ be a generalized virtual decuspidaloid; $O$ a non-proper object of $\mf{X}$. Then $(\mathfrak{X}|_{O},F|_{O})$ is a generalized virtual decuspidaloid, where we note that the isomorphism of \cref{defi_gen_virtual_decuspidaloid} is obtained by forming the quotient by the ``$S$'' that appears in the pair ``$(H,S)$'' of \cref{defi_virtual_decuspidaloid} corresponding to $O$. 
\end{coro}

\begin{defi}
    Let $(\mathfrak{X},F)$ be a generalized virtual decuspidaloid; $O$, $O'$ objects of $\mf{X}$. Assume that the set of morphisms from $O$ to $O'$ is nonempty. Then denote by $m^{O,O'}$ the unique morphism from $O$ to $O'$ (cf. \cref{defi_virtual_decuspidaloid}).
\end{defi}


\section{Generic Symmetries of Configuration Spaces}\label{section_symmetry}

In this section, we study generic symmetries of 2-configuration spaces and examine automorphisms of inclusions of CAVC-type arising from these spaces.

\begin{lemm}\label{normalizer}
    Let $\mathcal{C} = (X,Y,f,k,\bar{k},t,s)$ be a pointed virtual curve over a field $k$ of characteristic zero. Let $I$ be a representative of an element of $\cusp(\mathcal{C})$. Then the normalizer $N_{\Delta_{\mathcal{C}}}(I)$ of $I$ in $\Delta_{\mathcal{C}}$ is $I$. Moreover, if $I'$ is another representative of an element of $\cusp(\mathcal{C})$ such that $I'\neq I$, then $I\cap I' = \{1\}$.
\end{lemm}
\begin{proof}
    By [\cite{SZM}, Proposition 4.20(ii)], one may assume without loss of generality that $s$ arises from a rational point. Then \cref{normalizer} follows immediately from [\cite{SZM}, Proposition 4.20(iii); \cite{CombGC}, Proposition 1.2(i)(ii)].
\end{proof}

\begin{prop}\label{generic_symmetry_genus_0}

    Let $k$ be a field of characteristic zero; $X$ a smooth curve over $k$ that is isomorphic to $ \mathbf{P}^{1}_{k}\setminus\{0,1,\infty\}$. Then the following assertions hold:

    (i) The subgroup $I$ of the group of $k$-automorphisms $\Aut_{k}(X_{2})$ of $X_{2}$ that consists of $f\in \Aut_{k}(X_{2})$ such that there exists a $k$-automorphism $h\in \Aut_{k}(X)$ of $X$ and a commutative diagram
    \begin{displaymath}
        \xymatrix{
            X_{2}\ar^{f}[r]\ar^{\pr^{2/1}_{X}}[d] &X_{2}\ar^{\pr^{2/1}_{X}}[d]\\
            X\ar^{h}[r]& X
        }
    \end{displaymath}
    is isomorphic to $\mathbb{S}_{4}$. Moreover, for any $f\in I$, the $h\in \Aut_{k}(X)$ that satisfies the above commutativity property is unique. In particular, we obtain a homomorphism $u:I\to \Aut_{k}(X)$, whose kernel $J$ is isomorphic to $\Z/2\Z\times \Z/2\Z$. 
    
    (ii) We maintain the notations and terminologies of assertion (i). Assume further that $s:G_{k}\to \Pi_{X}$ is a section of the natural surjection $\Pi_{X}\surjto G_{k}$. Then each element of $J$ may be regarded as an automorphism of $[\pr^{2/1}_{X},s]$. Therefore, $J$ acts  naturally on $\cusp([\pr^{2/1}_{X},s])$. This action of $J$ on $\cusp([\pr^{2/1}_{X},s])$ is faithful. Moreover, the image of $J$, regarded as a subgroup of $\Aut_{\set}(\cusp([\pr^{2/1}_{X},s]))\simto \mathbb{S}_{4}$, is the Klein group (i.e., the intersection of all $2$-Sylow subgroups of $\mathbb{S}_{4}$).
\end{prop}
\begin{proof}
    \cref{generic_symmetry_genus_0} follows immediately from [\cite{SZM}, Proposition 4.20(ii)(iii); \cite{SZM}, Lemma 6.4(v)(vi)(vii)].
\end{proof}



\begin{prop}\label{generic_symmetry_genus_1}
    Let $k$ be a field of characteristic zero; $X$ a once-punctured elliptic curve over $k$ (i.e., a hyperbolic curve over $k$ of type $(1,1)$); $Y$ the unique smooth compactification of $X$ up to canonical isomorphism. Regard $Y$ as an abelian variety over $k$ by taking the unique $k$-rational cusp of $X$ to be the identity. Then the following assertions hold:

    (i) The natural group homomorphism $\Delta_{X}\to \Delta_{Y}$ induced by the open immersion $X\injto Y$ induces an isomorphism
    \begin{displaymath}
        (\Delta_{X})^{\abt}\simto \Delta_{Y}
    \end{displaymath}
    (cf. \cref{Notations and Terminologies}, the discussion entitled ``Groups and Topologies'').
    In particular, $\Delta_{Y}$ is abelian and torsion-free. Moreover, the open immersion $X\injto Y$ induces an isomorphism
    \begin{displaymath}
        \Delta_{X}/\cusp(X)\simto \Delta_{Y}
    \end{displaymath}
    (cf. \cref{Notations and Terminologies}, the discussion entitled ``Groups and Topologies'').

    (ii) Let $f:X\to X$ be an automorphism of the $k$-scheme $X$. Assume that the restriction $f_{\ast}|_{(\Delta_{X})^{\abt}}$ of $f_{\ast}$ to the torsion-free abelianization $(\Delta_{X})^{\abt}$ of the geometric fundamental group is the identity, up to an inner automorphism of $\Pi_{X}$. Then $f$ is the identity.

    (iii) Let $f:Y\to Y$ be an automorphism of the $k$-scheme $Y$. Then $f$ is a translation (cf. \cref{Notations and Terminologies}, the discussion entitled ``Fields, Schemes and Curves'') if and only if the restriction $f_{\ast}|_{\Delta_{Y}}$ of $f_{\ast}$ to the geometric fundamental group $\Delta_{Y}$ is the identity, up to an inner automorphism of $\Pi_{Y}$.

    (iv) Let $a\in Y(k)$, $f:Y\to Y$ the automorphism of $Y$ defined by $x\mapsto a-x$. Then up to an inner automorphism of $\Pi_{Y}$, the restriction $f_{\ast}|_{\Delta_{Y}}$ of $f$ to the (abelian) geometric fundamental group coincides with the map $[-1]|_{\Delta(Y)}$ given by multiplication by $-1$.

    (v) The $2$-configuration space $X_{2}$ may be naturally regarded as an open subscheme of $Y\times_{k}Y$. Consider the automorphism $\sigma$ of the $k$-scheme $Y\times_{k}Y$ defined by $(x,y)\mapsto (x,x-y)$. Then $\sigma$ is of order $2$. Moreover, $\sigma$ induces an automorphism $\tau$ of $X_{2}$ via the natural open immersion $X_{2}\injto Y\times_{k}Y$ such that $\pr^{2/1}_{X}\circ \tau = \pr^{2/1}_{X}$.
    
    (vi) We maintain the notation of assertion (v). Assume further that $s:G_{k}\to \Pi_{X}$ is a section. Then $\tau$ may be regarded as an automorphism of $[\pr^{2/1}_{X},s]$, hence acts naturally on $\cusp([\pr^{2/1}_{X},s])$. This action of $\tau$ on $\cusp([\pr^{2/1}_{X},s])$ (which is a set of cardinality $2$) permutes the two elements.

    (vii) We maintain the notation of assertion (vi). Denote by $\delta_{\tau}:\Delta_{[\pr^{2/1}_{X},s]}\to \Delta_{[\pr^{2/1}_{X},s]}$ the automorphism of the geometric virtual fundamental group induced by $\tau$. Then, up to an inner automorphism of $\Pi_{[\pr^{2/1}_{X},s]}$, the restriction
    \begin{displaymath}
        \delta_{\tau}|_{(\Delta_{[\pr^{2/1}_{X},s]}/\cusp([\pr^{2/1}_{X},s]))^{\abt}}
    \end{displaymath}
    of $\delta_{\tau}$ to the torsion-free abelianization $(\Delta_{[\pr^{2/1}_{X},s]}/\cusp([\pr^{2/1}_{X},s]))^{\abt}$ of the quotient $\Delta_{[\pr^{2/1}_{X},s]}/\cusp([\pr^{2/1}_{X},s])$ (cf. \cref{Notations and Terminologies}, the discussion entitled ``Groups and Topologies'') coincides with the map 
    \begin{displaymath}
        [-1]_{(\Delta_{[\pr^{2/1}_{X},s]}/\cusp([\pr^{2/1}_{X},s]))^{\abt}}.
    \end{displaymath} 
    given by multiplication by $-1$.
\end{prop}
\begin{proof}
    Assertion (i) follows immediately from well-known properties of the corresponding topological fundamental group (cf. also [\cite{CanLift}, Proposition 2.3(ii)]).

    Next, we consider assertion (ii). Although [\cite{SZM}, Proposition 4.6] only states that the natural isomorphism $\Pi_{A_{\bar{k}}}\simto \T(A_{\bar{k}})$ is compatible with the $G_{k}$-action, the proof in fact implies that this isomorphism is functorial with respect to arbitrary automorphisms of abelian varieties. It thus follows that $f$ acts trivially on the set of all torsion points of $X$. Consequently, since $X$ is a curve, we conclude that $f$ is the identity.

    Next, we consider assertion (iii). Note that the elements of the geometric fundamental group of an abelian variety arise from translations (cf. the proof of [\cite{SZM}, Proposition 4.6]). Consequently, since abelian varieties are commutative, any translation induces the identity map on $\Delta_{Y}$, up to an inner automorphism of $\Pi_{Y}$. 
    Thus, after applying a suitable translation, we may assume that $f$ stabilizes the identity element of $Y$. Consequently, assertion (iii) follows immediately from assertions (i) and (ii).

    Assertion (iv) follows immediately, in light of assertion (iii), from [\cite{SZM}, Proposition 4.6], as discussed in the proof of assertion (ii).
    
    Assertions (v) and (vi) follow immediately from well-known basic properties of abelian varieties. 
    
    Finally, we consider assertion (vii). By [\cite{SZM}, Theorem 4.20(ii)], we may assume without loss of generality that $s$ arises from a $k$-rational point $a$ of $X$. Regard $a$ as a $k$-rational point of $Y$. Then, by [\cite{SZM}, Corollary 3.8] and [\cite{SZM}, Theorem 4.20(iii)], assertion (vii) follows immediately from assertions (i) and (iv).



\end{proof}


\begin{theo}\label{reconstruction_of_section}
    Let $k$ be an AF; $X$ a hyperbolic curve over $k$; $s:G_{k}\to \Pi_{X}$ a section. Write $\diag_{X}\in \cusp([\pr^{2/1}_{X},s])$ for the conjugacy class of cuspidal inertia subgroups determined by the decuspidalization $[\pr^{2/1}_{X},s]\to [\pr^{1}_{X},s]$ of pointed virtual curves; $Y$ for the unique smooth compactification of $X$ up to canonical isomorphism; $D\subset X\times_{k}Y$ for the diagonal (i.e., the reduced closed subscheme determined by the image of $\id_{X}\times_{k}i:X\to X\times_{k}Y$, where $\id_{X}:X\simto X$ denotes the identity morphism, and $i:X\injto Y$ denotes the natural open immersion). Let $I\subset \Pi_{X_{2}}$ be a representative of $\diag_{X}$ in $I(\pr^{2/1}_{X},\Pi_{X_{2}})$, where we regard $\diag_{X}$ as an element of $\cusp(\pr^{2/1}_{X})$ (cf. \cref{Notations and Terminologies}, the discussion entitled ``Fundamental Groups''; \cref{virtual_inertia_and_real_inertia}). Write $J$ for the element of $D(\pr^{2/1}_{X},\Pi_{X_{2}})$ corresponding to $I$ (cf. \cref{Notations and Terminologies}, the discussion entitled ``Fundamental Groups''; \cref{virtual_inertia_and_real_inertia}). Then the following assertions hold:

    (i) The schemes $X_{2}$ and $X\times_{k}X$ may be regarded as open subschemes of $X\times_{k}Y$ via the natural open immersion $X\injto Y$. Moreover, the natural open immersion $X_{2}\injto X\times_{k}Y$ (respectively, $X\times_{k}X\injto X\times_{k}Y$) is the compactification of the family $\pr^{2/1}_{X}$ (respectively, $\pr^{1}_{X}$) of hyperbolic curves (cf. \cref{Notations and Terminologies}, the discussion entitled ``Fields, Schemes and Curves'').

    (ii) The normalizer $N_{\Delta_{[\pr^{2/1}_{X},s]}}(I)$ of $I$ in $\Delta_{[\pr^{2/1}_{X},s]}$ coincides with $I$.

    (iii) The group $J$ normalizes $I$.

    (iv) The morphisms $(\pr^{1}_{X})_{\ast}$ and $(\pr^{2}_{X})_{\ast}$ induce a natural outer isomorphism $\Delta_{X\times_{k}X}\simto \Delta_{X}\times \Delta_{X}$.

    (v) The group homomorphism $(\pr^{2/1}_{X})_{\ast}$ induces an outer isomorphism $J/I\simto\Pi_{X}$ (cf. (iii)). In particular, the subgroup $J\cap \Pi_{[\pr^{2/1}_{X},s]}$ maps surjectively onto $G_{k}$ under the natural surjection $\Pi_{[\pr^{2/1}_{X},s]}\surjto G_{k}$.

    (vi) The image $J'$ of $J$ in $\Pi_{X\times_{k}X}$, obtained via the natural open immersion $X_{2}\injto X\times_{k}X$, coincides, up to an inner automorphism of $\Pi_{X\times_{k}X}$, with the equalizer of 
    \begin{displaymath}
	\xymatrix{
		\Pi_{X\times_{k}X}\ar@<1ex>[r]^>>>>>{(\pr^{1}_{X})_{\ast}}\ar@<-1ex>[r]_>>>>>{(\pr^{2}_{X})_{\ast}}& \Pi_{X}
	},
    \end{displaymath}
    where, by a slight abuse of notation, ``$(\pr^1_X)_\ast$'' and ``$(\pr^2_X)_\ast$'' denote representatives of the 
    respective outer homomorphisms.
    Moreover, $(\pr^{1}_{X})_{\ast}$ and $(\pr^{2}_{X})_{\ast}$ induce outer isomorphisms $J'\simto\Pi_{X}$.

    (vii) The normalizer $N_{\Pi_{[\pr^{2/1}_{X},s]}}(I)$ of $I$ in $\Pi_{[\pr^{2/1}_{X},s]}$ coincides with $J\cap \Pi_{[\pr^{2/1}_{X},s]}$.

    (viii) Denote by $H\subset \Pi_{[\pr^{1}_{X},s]}$ the image of $N_{\Pi_{[\pr^{2/1}_{X},s]}}(I)$ in $\Pi_{[\pr^{1}_{X},s]}$ via the natural morphism of virtual fundamental groups induced by the decuspidalization $[\pr^{2/1}_{X},s]\to [\pr^{1}_{X},s]$ of pointed virtual curves (cf. (i)). Then the natural surjection $\Pi_{[\pr^{1}_{X},s]}\surjto G_{k}$ maps $H$ isomorphically onto $G_{k}$. Hence $H$ defines a Galois section. Note that $\Pi_{[\pr^{1}_{X},s]}$ is naturally isomorphic to $\Pi_{X}$ via $(\pr^{2}_{X})_{\ast}$ (cf. [\cite{SZM} Lemma 5.11(i)]). Therefore, one may regard $H$ as a section $s':G_{k}\injto \Pi_{X}$ of the natural surjection $\Pi_{X}\surjto G_{k}$. Then the section $s'$ coincides with the 
    section $s$ up to $\Pi_X$-conjugation.
\end{theo}
\begin{proof}
    Assertion (i) follows immediately from the various definitions involved. 
    Assertion (ii) follows immediately from \cref{normalizer}.
    Assertion (iii) follows immediately from the definitions of the notions of decomposition and inertia groups. Assertion (iv) follows immediately from [\cite{SGA1}, Expos\'e XIII, Proposition 4.6].
    Assertions (v) and (vi) follow immediately from the various definitions involved.
    Next, we consider assertion (vii). By assertion (iii), $J$ normalizes $I$. Consequently, $J\cap \Pi_{[\pr^{2/1}_{X},s]}$ also normalizes $I$. Hence, assertion (vii) follows immediately from assertions (ii) and (v).
    Finally, we consider assertion (viii). By assertion (vi), without loss of generality, we may choose $I$ such that the image $J'$ of $J$ in $\Pi_{X\times_{k}X}$ is the equalizer of $(\pr^{1}_{X})_{\ast}$ and $(\pr^{2}_{X})_{\ast}$. 
    By assertion (vii), $J'\cap \Pi_{[\pr^{1}_{X},s]} = H$. Recall that $(\pr^{2}_{X})_{\ast}$ induces a natural isomorphism $\Pi_{[\pr^{1}_{X},s]}\simto \Pi_{X}$. Regard $H$ as a subgroup of $\Pi_{X\times_{k}X}$. By definition, $(\pr^{1}_{X})_{\ast}(H)$ is equal to $s(G_{k})$ in $\Pi_{X}$. Assertion (viii) then follows immediately from the fact that $H$ is a subgroup of $J'$, and that $J'$ is the equalizer of $(\pr^{1}_{X})_{\ast}$ and $(\pr^{2}_{X})_{\ast}$.
\end{proof}

\begin{coro}\label{reconstruction_of_section_special_sym}
    Let $k^{\dagger}$, $k^{\ddagger}$ be AFs; $\bar{k}^{\dagger}$ (respectively, $\bar{k}^{\ddagger}$) an algebraic closure of $k^{\dagger}$ (respectively, $k^{\ddagger}$); $G_{k^{\dagger}}$ (respectively, $G_{k^{\ddagger}}$) the Galois group $\gal(\bar{k}^{\dagger}/k^{\dagger})$ (respectively, $\gal(\bar{k}^{\ddagger}/k^{\ddagger})$); $X^{\dagger}$ (respectively, $X^{\ddagger}$) a hyperbolic curve over $k^{\dagger}$ (respectively, $k^{\ddagger}$) of type $(g^{\dagger},r^{\dagger})$ (respectively, $(g^{\ddagger},r^{\ddagger})$); $s^{\dagger}:G_{k^{\dagger}}\rightarrow \Pi_{X^{\dagger}}$ (respectively, $s^{\ddagger}:G_{k^{\ddagger}}\rightarrow \Pi_{X^{\ddagger}}$) a section of the natural surjection $\Pi_{X^{\dagger}}\surjto G_{k^{\dagger}}$ (respectively, $\Pi_{X^{\ddagger}}\surjto G_{k^{\ddagger}}$). Write
    $\diag_{X^{\dagger}}\in \cusp([\pr^{2/1}_{X^{\dagger}}, s^{\dagger}])$ (respectively, $\diag_{X^{\ddagger}}\in \cusp([\pr^{2/1}_{X^{\ddagger}}, s^{\ddagger}])$) for the element determined by the decuspidalization $[\pr^{2/1}_{X^{\dagger}}, s^{\dagger}]\to [\pr^{1}_{X^{\dagger}}, s^{\dagger}]$ (respectively, $[\pr^{2/1}_{X^{\ddagger}}, s^{\ddagger}]\to [\pr^{1}_{X^{\ddagger}}, s^{\ddagger}]$). 
    Let $\theta: \Pi_{[\pr^{2/1}_{X^{\dagger}},s^{\dagger}]}\simto \Pi_{[\pr^{2/1}_{X^{\ddagger}},s^{\ddagger}]}$ be an isomorphism of profinite groups. Then the following assertions hold:

    (i) The type $(g^{\dagger}, r^{\dagger})$ coincides with the type $(g^{\ddagger}, r^{\ddagger})$. Moreover, $\theta$ induces an 
    isomorphism $G_{k^{\dagger}}\simto G_{k^{\ddagger}}$, as well as a bijection $\cusp([\pr^{2/1}_{X^{\dagger}}, s^{\dagger}])\simto \cusp([\pr^{2/1}_{X^{\ddagger}}, s^{\ddagger}])$.

    (ii) Suppose that the action induced by $\theta$ on the conjugacy classes of cuspidal inertia subgroups maps $\diag_{X^{\dagger}}$ to $\diag_{X^{\ddagger}}$. Then, after possibly composing $\theta$ with an inner automorphism 
    of its domain/codomain, $\theta$ induces an isomorphism $\alpha:\Pi_{X^{\dagger}}\simto \Pi_{X^{\ddagger}}$ that arises from an isomorphism $X^{\dagger}\simto X^{\ddagger}$, and, moreover, satisfies the equality $\alpha(s^{\dagger}(G_{k^{\dagger}}))=s^{\ddagger}(G_{k^{\ddagger}})$.

    (iii) Suppose that $(g^{\dagger}, r^{\dagger}) = (0,3)$ and that the natural action of $\Pi_{[\pr^{2/1}_{X^{\dagger}},s^{\dagger}]}$ on $\cusp([\pr^{2/1}_{X^{\dagger}},s^{\dagger}])$ is trivial. 
    Then, after possibly composing $\theta$ with an automorphism arises from an element in the group ``$J$'' as in \cref{generic_symmetry_genus_0}(ii),
    $\theta$ induces an isomorphism $\alpha:\Pi_{X^{\dagger}}\simto \Pi_{X^{\ddagger}}$ that arises from an isomorphism $X^{\dagger}\simto X^{\ddagger}$, and, moreover, satisfies the equality $\alpha(s^{\dagger}(G_{k^{\dagger}}))=s^{\ddagger}(G_{k^{\ddagger}})$.

    (iv) Suppose that $(g^{\dagger}, r^{\dagger}) = (1,1)$. 
    Then, after possibly composing $\theta$ with an automorphism ``$\tau$'' as in \cref{generic_symmetry_genus_1}(vi), $\theta$ induces an isomorphism $\alpha:\Pi_{X^{\dagger}}\simto \Pi_{X^{\ddagger}}$ that arises from an isomorphism $X^{\dagger}\simto X^{\ddagger}$, and, moreover, satisfies the equality $\alpha(s^{\dagger}(G_{k^{\dagger}}))=s^{\ddagger}(G_{k^{\ddagger}})$.

    (v) Suppose that $(g^{\dagger}, r^{\dagger}) = (g,0)$, where $g\geq 2$. 
    Then, after possibly composing $\theta$ with an inner automorphism 
    of its domain/codomain, $\theta$ induces an isomorphism $\alpha:\Pi_{X^{\dagger}}\simto \Pi_{X^{\ddagger}}$ that arises from an isomorphism $X^{\dagger}\simto X^{\ddagger}$, and, moreover, satisfies the equality $\alpha(s^{\dagger}(G_{k^{\dagger}}))=s^{\ddagger}(G_{k^{\ddagger}})$.
\end{coro}
\begin{proof}
    Assertion (i) follows immediately from \cref{CAVC_universal_prop}, \cref{cls_diff_uncls}(i)(ii), and [\cite{SZM}, Theorem 5.1; \cite{SZM}, Theorem 5.13; \cite{NSW}, Theorem 7.4.1]. 
    Next, we consider assertion (ii). By construction, the morphism $\theta$ induces an isomorphism
    \begin{displaymath}
        \theta':\Pi_{[\pr^{2/1}_{X^{\dagger}},s^{\dagger}]}/\{\diag_{X^{\dagger}}\} \simto  \Pi_{[\pr^{2/1}_{X^{\ddagger}},s^{\ddagger}]}/\{\diag_{X^{\ddagger}}\}.
    \end{displaymath}
    which in turn induces an isomorphism
    \begin{displaymath}
        \theta'':\Pi_{[\pr^{1}_{X^{\dagger}},s^{\dagger}]} \simto  \Pi_{[\pr^{1}_{X^{\ddagger}},s^{\ddagger}]}.
    \end{displaymath}
    Observe that $[\pr^{1}_{X^{\dagger}},s^{\dagger}]$ and $[\pr^{1}_{X^{\ddagger}},s^{\ddagger}]$ are split pointed virtual curves that admit respective splitting morphisms to $X^{\dagger}$ and $X^{\ddagger}$ (cf. \cref{constructible_trivial_family}). Therefore, $\theta''$ may be regarded as an isomorphism
    \begin{displaymath}
        \theta''':\Pi_{X^{\dagger}} \simto  \Pi_{X^{\ddagger}}.
    \end{displaymath}
    By assertion (i) and \cref{anabelian_CAVC}(i), this yields an isomorphism $X^{\dagger}\simto X^{\ddagger}$.
    Let $I\subset \Delta_{[\pr^{2/1}_{X^{\dagger}},s^{\dagger}]}$ be a representative of $\diag_{X^{\dagger}}$. Then $\theta(I)$ serves as a representative of $\diag_{X^{\ddagger}}$. Furthermore, we have
    \begin{displaymath}
        \theta\big(N_{\Pi_{[\pr^{2/1}_{X^{\dagger}},s^{\dagger}]}}(I)\big) = N_{\Pi_{[\pr^{2/1}_{X^{\ddagger}},s^{\ddagger}]}}(\theta(I)).
    \end{displaymath}
    Denote by $H$ (respectively, $H'$) the natural image of $I$ (respectively, $\theta(I)$) in $\Pi_{X^{\dagger}}$ (respectively, $\Pi_{X^{\ddagger}}$). Then $\theta'''(H) = H'$. Therefore, the remaining part of assertion (ii) follows immediately from \cref{reconstruction_of_section}(viii). 
    Assertion (iii) follows immediately from assertion (ii), \cref{generic_symmetry_genus_0}(ii), and [\cite{SZM}, Proposition 6.8(iv)].
    Assertion (iv) follows immediately from assertion (ii) and \cref{generic_symmetry_genus_1}(vi).
    Assertion (v) follows immediately from assertion (ii), by observing that $\diag_{X^{\dagger}}$ (respectively, $\diag_{X^{\ddagger}}$) is the unique element of $\cusp([\pr^{2/1}_{X^{\dagger}}, s^{\dagger}])$ (respectively, $\cusp([\pr^{2/1}_{X^{\ddagger}}, s^{\ddagger}])$).
\end{proof}

\begin{lemm}\label{devissage_of_section}
    Let $k$ be an AF; $X$ a hyperbolic curve over $k$; $s:G_{k}\injto \Pi_{X}$ a Galois section. Then the following assertions hold:

    (i) Let $k'$ be a finite field extension of $k$. Suppose that the restriction $s':G_{k'}\injto \Pi_{X\times_{k}k'}$ of $s$ to $G_{k'}$ arises from a $k'$-rational point. Then $s$ arises from a $k$-rational point.

    (ii) Suppose that $f:X\to Y$ is an open immersion, where $Y$ is a hyperbolic curve over $k$. Assume further that the composite $f_{\ast}\circ s$ arises from a $k$-rational point. Then $s$ arises from a $k$-rational point.
\end{lemm}
\begin{proof}
    Assertion (i) follows from the injectivity portion of the Section Conjecture, e.g., as given in [\cite{MZK8}, Theorem C].
    Assertion (ii) follows from [\cite{RNSPM}, Theorem F(ii)], together with the discussion of the number field case immediately following the statement of [\cite{RNSPM}, Theorem F(ii)].
\end{proof}

\begin{theo}\label{thmb_simp}
    Let $k$ (respectively, $k'$) be an AF; $\bar{k}$ (respectively, $\bar{k}'$) an algebraic closure of $k$ (respectively, $k'$); $G_{k}$ (respectively, $G_{k'}$) the Galois group $\gal(\bar{k}/k)$ (respectively, $\gal(\bar{k}'/k')$); $X$ (respectively, $X'$) a hyperbolic curve over $k$ (respectively, $k'$) of type $(0,3)$ (respectively, type $(1,3)$); $s:G_{k}\rightarrow \Pi_{X}$ (respectively, $s':G_{k'}\rightarrow \Pi_{X'}$) a section of the natural surjection $\Pi_{X}\surjto G_{k}$ (respectively, $\Pi_{X'}\surjto G_{k'}$). Write $\Pi$ (respectively, $\Pi'$) for the virtual fundamental group $\Pi_{[\pr^{2/1}_{X},s]}$ (respectively, $\Pi_{[\pr^{2/1}_{X'},s']}$); $\Delta$ (respectively, $\Delta'$) for the geometric virtual fundamental group $\Delta_{[\pr^{2/1}_{X},s]}$ (respectively, $\Delta_{[\pr^{2/1}_{X'},s']}$). Suppose that there exists a group injection $f:\Pi'\injto \Pi$ of index $2$ such that $f^{-1}(\Delta)=\Delta'$ (where we note that in the case where $k$ is an NF, the assumption $f^{-1}(\Delta)=\Delta'$ is satisfied automatically by [\cite{SZM}, Theorem 5.1]), and that the induced injection $f_{\Delta}:\Delta'\injto \Delta$ is of index $2$. Assume further that $X\simto \mathbb{P}^{1}_{k}\setminus\{0,1,\infty\}$, and that the automorphisms of $\Pi$ induced by elements of the group ``$J$'' appearing in \cref{generic_symmetry_genus_0}(ii) stabilize the image of $f$. Then the following assertions hold:

    (i) $s'$ arises from a $k'$-rational point.

    (ii) The inclusion $\Delta'\injto \Pi'$ is constructible (cf. \cref{defi_constructible_inclusion}(ii)).

    (iii) The inclusion $\Delta\injto \Pi$ is constructible.

    (iv) $s$ arises from a $k$-rational point.
\end{theo}
\begin{proof}
    First, we consider assertion (i). Write $(\mf{X},F)\coloneqq(\cavc_{\Delta\subset \Pi}, F_{\Delta\subset \Pi})$ for the virtual decuspidaloid associated to $\Delta\subset \Pi$ (cf. \cref{CAVC_functor_defi}). 
    Denote by $O$ the covering of $\root(\mf{X})$ induced by the open subgroup $f(\Pi')\subset \Pi$, by $T:O\to \root(\mf{X})$ the corresponding morphism of cov-type, by $T_\cusp:\cusp(O)\to \cusp(\root(\mf{X}))$ the map determined by $F$ (cf. \cref{determined_cusp_morphism}(i)).
    Since $T_\cusp$ is surjective (cf. \cref{covering_group_theoretic}; \cref{defi_cusp_compatibility_subgroup}), by construction (where we note that both $\cusp(O)$ and $\cusp(\root(\mf{X}))$ are of cardinality $4$), $T_\cusp$ is bijective.
    Denote by $\diag_{X}$ (respectively, $\diag_{X'}$) the element of $\cusp(\root(\mf{X}))$ (respectively, $\cusp(O)$) determined by the decuspidalization $[\pr^{2/1}_{X},s]\to [\pr^{1}_{X},s]$ (respectively, $[\pr^{2/1}_{X'},s']\to [\pr^{1}_{X'},s']$). Denote the remaining three elements of $\cusp(O)$ other than $\diag_{X'}$ by $\alpha$, $\beta$, and $\gamma$.
    Denote the decuspidalization of $O$ induced by $\{\diag_{X'}\}$ (respectively, $\{\alpha\}$; $\{\diag_{X'},\alpha\}$) by $O_{\{\diag_{X'}\}}$ (respectively, $O_{\{\alpha\}}$; $O_{\{\diag_{X'},\alpha\}}$).
    Let $I\subset f(\Pi')$ be a representative of $\diag_{X'}$. Denote by $N$ the normalizer of $I$ in $f(\Pi')$; by $I_{\{\diag_{X'}\}}$ (respectively, $I_{\{\alpha\}}$; $I_{\{\diag_{X'},\alpha\}}$) the natural image of $I$ in $\Pi(O_{\{\diag_{X'}\}})$ (respectively, $\Pi(O_{\{\alpha\}})$; $\Pi(O_{\{\diag_{X'},\alpha\}})$); by $N_{\{\diag_{X'}\}}$ (respectively, $N_{\{\alpha\}}$; $N_{\{\diag_{X'},\alpha\}}$) the natural image of $N$ in $\Pi(O_{\{\diag_{X'}\}})$ (respectively, $\Pi(O_{\{\alpha\}})$; $\Pi(O_{\{\diag_{X'},\alpha\}})$). By abuse of language, we regard $\cusp(O_{\diag_{X'}})$ (respectively, $\cusp(O_{\{\alpha\}})$; $\cusp(O_{\{\diag_{X'},\alpha\}})$) as a subset of $\cusp(O)$ (cf. \cref{determined_cusp_morphism}(ii)), hence their elements may be regarded as elements of $\cusp(O)$.
    By [\cite{SZM}, Proposition 6.1], $O_{\{\diag_{X'}\}}$ is constructible. 
    By assumption, the automorphisms of $\Pi$ induced by elements of ``$J$'' described in \cref{generic_symmetry_genus_0}(ii) naturally induce an action of $J$ on the inclusion $F(O)$. 
    By construction (cf. \cref{determined_cusp_morphism}(i)) and \cref{generic_symmetry_genus_0}(ii), together with the fact that $T_\cusp$ is bijective, there exists an element $\theta\in J$ such that the action $\theta|_{F(O)}$ of $\theta$ on $F(O)$ permutes the two elements $\diag_{X'}$, $\alpha$ (cf. [\cite{SZM}, Theorem 4.20(ii)]). Therefore, $\theta$ induces an isomorphism $F(O_{\{\diag_{X'}\}})\simto F(O_{\{\alpha\}})$. Thus, $O_{\{\alpha\}}$ is constructible. 
    By \cref{anabelian_CAVC_more}(ii), $O_{\{\diag_{X'},\alpha\}}$ is constructible.
    By \cref{reconstruction_of_section}(viii), $N_{\{\diag_{X'}\}}$ defines a Galois section of $\Pi(O_{\{\diag_{X'}\}})\surjto G(O_{\{\diag_{X'}\}})$. Therefore, $N_{\{\diag_{X'},\alpha\}}$ defines a Galois section of $\Pi(O_{\{\diag_{X'},\alpha\}})\surjto G(O_{\{\diag_{X'},\alpha\}})$. In particular, $N_{\{\diag_{X'},\alpha\}}$ surjects onto $G(O_{\{\diag_{X'},\alpha\}})$, hence $N$ (respectively, $N_{\{\diag_{X'}\}}$; $N_{\{\alpha\}}$) surjects onto $G(O)$ (respectively, $G(O_{\{\diag_{X'}\}})$; $G(O_{\{\alpha\}})$).
    By construction (cf. \cref{decusp_universal}), $I_{\{\alpha\}}$ is a representative of $\diag_{X'}$ (where we regard $\diag_{X'}$ as an element of $\cusp(O_{\{\alpha\}})$), and $N_{\{\alpha\}}$ normalizes $I_{\{\alpha\}}$. Since $N_{\{\alpha\}}$ surjects onto $G(O_{\{\alpha\}})$, and $O_{\{\alpha\}}$ is constructible, it follows from \cref{normalizer} that $N_{\{\alpha\}}$ is the normalizer of $I_{\{\alpha\}}$. Thus, since $O_{\{\alpha\}}$ is constructible, it follows that
    $N_{\{\diag_{X'},\alpha\}}$ is a Galois section that arises from a rational point. In particular, by \cref{devissage_of_section}(ii), $N_{\{\diag_{X'}\}}$ is a Galois section that arises from a rational point. Thus, by \cref{reconstruction_of_section}(viii), $s'$ arises from a $k'$-rational point.  This completes the proof of assertion (i).
    Assertion (ii) follows immediately from assertion (i) and [\cite{SZM}, Corollary 3.8]. 
    
    Next, we consider assertion (iii). Denote by $Y$ the canonical scheme associated to $\Delta'\subset \Pi'$ (cf. \cref{anabelian_CAVC}).  One verifies immediately that $Y$ is a curve of type $(1,4)$. By \cref{anabelian_CAVC}(ii), in order to establish assertion (iii), it suffices to verify that the induced action of the unique nontrivial element $\omega$ of ($\Delta/\Delta'\simto$) $\Pi/\Pi'$ on $Y$ does not fix any closed point of $Y$. Denote by $Y^{\cl}$ the unique compactification of $Y$. By construction, the induced action of $\omega$ on $Y^{\cl}$ already fixes the four $k'$-rational points of $Y^{\cl}$ that arise from the four cusps of $Y$. Therefore, assertion (iii) follows immediately from the Hurwitz formula.
    Assertion (iv) follows immediately from assertion (iii) and \cref{reconstruction_of_section}(viii).
\end{proof}

\begin{coro}\label{real_thmb}
    Let $k$ (respectively, $k'$) be an AF; $\bar{k}$ (respectively, $\bar{k}'$) an algebraic closure of $k$ (respectively, $k'$); $G_{k}$ (respectively, $G_{k'}$) the Galois group $\gal(\bar{k}/k)$ (respectively, $\gal(\bar{k}'/k')$); $X$ (respectively, $X'$) a hyperbolic curve over $k$ (respectively, $k'$) of type $(0,3)$ (respectively, type $(1,3)$); $s:G_{k}\rightarrow \Pi_{X}$ (respectively, $s':G_{k'}\rightarrow \Pi_{X'}$) a section of the natural surjection $\Pi_{X}\surjto G_{k}$ (respectively, $\Pi_{X'}\surjto G_{k'}$). Write $\Pi$ (respectively, $\Pi'$) for the virtual fundamental group $\Pi_{[\pr^{2/1}_{X},s]}$ (respectively, $\Pi_{[\pr^{2/1}_{X'},s']}$); $\Delta$ (respectively, $\Delta'$) for the geometric virtual fundamental group $\Delta_{[\pr^{2/1}_{X},s]}$ (respectively, $\Delta_{[\pr^{2/1}_{X'},s']}$). 
    Assume that there exists an open group injection $f:\Pi'\injto \Pi$ such that $f^{-1}(\Delta)=\Delta'$ (where we note that in the case where $k$ is an NF, the assumption $f^{-1}(\Delta)=\Delta'$ is satisfied automatically by [\cite{SZM}, Theorem 5.1]).
    Then the following assertions hold:

    (i) The induced injection $f_{\Delta}:\Delta'\injto \Delta$ is of index $2$.

    (ii) $s'$ arises from a $k'$-rational point.

    (iii) The inclusion $\Delta'\injto \Pi'$ is constructible (cf. \cref{defi_constructible_inclusion}(ii)).

    (iv) The inclusion $\Delta\injto \Pi$ is constructible.

    (v) $s$ arises from a $k$-rational point.
\end{coro}
\begin{proof}
    Assertion (i) follows immediately from [\cite{SZM}, Proposition 4.18(iii)] and the Hurwitz formula.
    By \cref{characterstic_of_1_4} and [\cite{SZM}, Proposition 6.8(iv)], one may assume that there exists a finite field extension $k''$ of $k$ such that the restriction $f':f^{-1}(\Pi'')\to \Pi''$ of $f$ to $\Pi''\coloneqq \Pi_{[\pr^{2/1}_{X\times_{k}k''},s|_{G_{k''}}]}$ satisfies the assumptions of \cref{thmb_simp}. Therefore, assertion (ii) (respectively, assertion (v)) follows immediately from \cref{devissage_of_section}(i) and \cref{thmb_simp}(i) (respectively, \cref{thmb_simp}(iv)), and assertion (iii) (respectively, assertion (iv)) follows immediately from \cref{anabelian_CAVC_more}(iii) and \cref{thmb_simp}(ii) (respectively, \cref{thmb_simp}(iii)).
\end{proof}

\begin{lemm}\label{characterstic_of_1_4}
    Let $k$ be an AF; $\bar{k}$ an algebraic closure of $k$; $G_{k}$ the Galois group $\gal(\bar{k}/k)$; $Z$ a hyperbolic curve over $k$ of type $(0,3)$ that is isomorphic over $k$ to $\mathbb{P}^{1}_{k}\setminus\{0,1,\infty\}$; $s:G_{k}\rightarrow \Pi_{Z}$ a section of the natural surjection $\Pi_{Z}\surjto G_{k}$. Write $\Pi$ for the virtual fundamental group $\Pi_{[\pr^{2/1}_{Z},s]}$; $\Delta$ for the geometric virtual fundamental group $\Delta_{[\pr^{2/1}_{Z},s]}$. Let $\Pi^{\dagger}$, $\Pi^{\ddagger}$ be two open subgroups of $\Pi$ such that the following conditions hold:
    \begin{enumerate}
        \item The subgroups $\Delta^{\dagger}\coloneqq \Pi^{\dagger}\cap \Delta$, $\Delta^{\ddagger}\coloneqq \Pi^{\ddagger}\cap \Delta$ are of index $2$ in $\Delta$.

        \item The types of $\Delta^{\dagger}\subset \Pi^{\dagger}$, $\Delta^{\ddagger}\subset \Pi^{\ddagger}$ are of $(1,4)$ (cf. \cref{defi_cusp_33}; \cref{covering_group_theoretic}).
    \end{enumerate}
    Then $\Delta^{\dagger} = \Delta^{\ddagger}$.
\end{lemm}
\begin{proof}
    Since $\Delta^{\dagger}$ (respectively, $\Delta^{\ddagger}$) is of index $2$ in $\Delta$, $\Delta^{\dagger}$ (respectively, $\Delta^{\ddagger}$) is normal in $\Delta$. Moreover, the quotient $\Delta/\Delta^{\dagger}$ (respectively, $\Delta/\Delta^{\ddagger}$) factors through the quotient $\Delta^{\ab}/2(\Delta^{\ab})$.
    By [\cite{SZM}, Theorem 4.20], without loss of generality, one may assume that $s$ arises from a $k$-rational point $z$ of $Z$. 
    Consequently, it is sufficient to verify the following claim:

    ($*$): Let $k$ be an algebraically closed field of characteristic zero; $X$ a hyperbolic curve over $k$ of type $(0,4)$; $f:Y\to X$ an \etale covering of degree $2$ such that $Y$ is of type $(1,4)$. Then the unique extension of $f$ to the smooth compactification $f^{\cl}:Y^{\cl}\to X^{\cl}$ is branched over each of the four cusps of $X^{\cl}$. Moreover, such a covering $f$ is unique up to isomorphism.

    We now proceed to prove the claim ($*$). First, it follows immediately from the Hurwitz formula that the extension $f^{\cl}:Y^{\cl}\to X^{\cl}$ is branched over each of the four cusps of $X^{\cl}$.
    Denote by $J$ the quotient $(\Pi_{X})^{\ab}/2(\Pi_{X})^{\ab}$. For each $i\in\{1,2,3,4\}$, let $I_{i}\subset J$ be the image of a representative of the conjugacy class associated to the $i$-th cusp of $X$, and let $e_{i}\in J$ denote the unique nontrivial element of $I_{i}$. Here, the indexing of the cusps is chosen arbitrarily. Let $H\subset J$ be the subgroup of index $2$ induced by the covering $Y\to X$.
    Since the extension $f^{\cl}:Y^{\cl}\to X^{\cl}$ is branched over each of the four cusps of $X^{\cl}$, it follows that $H$ satisfies the following conditions:
    \begin{enumerate}
        \item The subgroup $H\subset J$ is of index $2$ in $J$.
        \item For each $i\in\{1,2,3,4\}$, the intersection $H\cap I_{i}$ is trivial.
    \end{enumerate}
    It follows from a standard computation via Kummer theory for algebraic curves that $J$, regarded as a $\mathbb{Z}/2\mathbb{Z}$-vector space, has dimension $3$ and is generated by $\{e_{1},e_{2},e_{3},e_{4}\}$ subject to the single relation $e_{1}+e_{2}+e_{3}+e_{4}=0$. Therefore, by an elementary calculation in linear algebra, $H$ is uniquely determined. This completes the proof of \cref{characterstic_of_1_4}.
    
\end{proof}

\section{Reconstruction Algorithms}\label{section_reconstruction}

In this section, we discuss reconstruction algorithms associated to various generalized virtual decuspidaloids that arise from the virtual fundamental group $\Pi_{[\pr^{2/1}_{X},s]}$. In \cref{reconstruction_algorithm_genus_2}, we consider the case in which the genus is greater than or equal to $2$; in \cref{validity_of_descent_genus_1}, we consider the case in which the genus is $1$.

\begin{prop}\label{Galois_group_geometricity_from_open_subgroup} \quad

    (i) Let $G$ be a profinite group; $K$, $K'$ ACAFs; $f:G\to \Aut(K)$, $f':G\to \Aut(K')$ FT morphisms. Then $K$ and $K'$ are isomorphic (cf. \cref{Notations and Terminologies}, the discussion entitled ``Fields, Schemes and Curves'').

    (ii) Let $G$ be a profinite group; $K$ an ACAF; $f:G\to \Aut(K)$, $f':G\to \Aut(K)$ FT morphisms; $H\subset G$ an open subgroup; $f|_{H}:H\to \Aut(K)$, $f'|_{H}:H\to \Aut(K)$ the respective restrictions. Assume that $f|_{H}$ and $f'|_{H}$ are conjugate. Then $f$ and $f'$ are conjugate (cf. \cref{Notations and Terminologies}, the discussion entitled ``Fields, Schemes and Curves''). 
\end{prop}
\begin{proof}
    First, we consider assertion (i). Note that the characteristic prime of $K$ is non-zero if and only if $G$ is topologically finitely generated (cf. [\cite{MZK3}, Theorem 1.7(iii); \cite{NSW}, Theorem 7.4.1]). Thus, it suffices to treat the case in which we know a priori that the characteristic primes of $K$ and $K'$ are positive. Under this assumption, assertion (i) follows immediately from [\cite{MZK6}, Proposition 1.2.1(i)].  Assertion (ii) follows immediately from the slimness of $\Aut(K)$ (cf. [\cite{MZK3}, Theorem 1.7(ii)(iii)]).
\end{proof}


First, we consider the case where the genus is greater than or equal to $2$. This case admits a relatively straightforward treatment. Roughly speaking, it suffices to verify that the virtual fundamental group $\Pi_{[\pr^{2/1}_{X},s]}/\cusp([\pr^{2/1}_{X},s])$ arising from the decuspidalization of $[\pr^{2/1}_{X},s]$ at all cusps is naturally isomorphic to the \etale fundamental group $\Pi_{X^{\cl}}$ of the unique smooth compactification $X^{\cl}$ of $X$.

\begin{theo}\label{reconstruction_algorithm_genus_2}
    Let $k$ be an AF; $\bar{k}$ an algebraic closure of $k$; $G_{k}$ the Galois group $\gal(\bar{k}/k)$; $X$ a hyperbolic curve over $k$ of type $(g,r)$, where $g\geq 2$; $s:G_{k}\rightarrow \Pi_{X}$ a section of the natural surjection $\Pi_{X}\surjto G_{k}$. Write $\Pi$ for the virtual fundamental group $\Pi_{[\pr^{2/1}_{X},s]}$; $\Delta$ for the geometric virtual fundamental group $\Delta_{[\pr^{2/1}_{X},s]}$; $X^{\cl}$ for the unique smooth compactification of $X$; $(\mf{X},F)\coloneqq(\cavc_{\Delta\subset \Pi}, F_{\Delta\subset \Pi})$ for the virtual decuspidaloid associated to $\Delta\subset \Pi$ (cf. \cref{CAVC_functor_defi}). Then one may construct various geometric objects (i.e., $\base(\Pi)$, $\sche(\Pi)$, $K(\sche(\Pi))$, $f$, and $h$ as defined in the subsequent algorithms) from the abstract profinite group $\Pi$ functorially with respect to isomorphisms of topological groups by the following algorithm:

    (a) One constructs $(\mf{X},F)$ from the abstract group $\Pi$ functorially with respect to isomorphisms of topological groups (cf. [\cite{SZM}, Theorem 5.13]; \cref{construction_CAVC}).

    (b) Let $O\in \obj(\mf{X})$ be the unique proper (cf. \cref{defi_intrinsic_virtual_decuspidaloid}(v)) object such that $O$ is a decuspidalization of $\root(\mf{X})$ (cf. \cref{prop_intrinsic_virtual_decuspidaloid}(ii)). (Note that the genus of $\root(\mf{X})$ is greater than or equal to $2$. Hence, $\cusp(\root(\mf{X}))$ is of admissible-type (cf. \cref{defi_admissible_type}). Therefore, existence is immediate. Uniqueness follows immediately from \cref{prop_intrinsic_virtual_decuspidaloid}(v).) By \cref{anabelian_CAVC_more}(ii) and [\cite{SZM}, Proposition 6.1], $O$ is constructible (cf. \cref{defi_constructible_object}). Moreover, the canonical scheme $(\sch(\Delta(O)\subset \Pi(O)))^{\Pi(O)/\Delta(O)}$ is naturally isomorphic to $X^{\cl}$ (cf. \cref{anabelian_CAVC}(i)(ii)).

    (c) Since $O$ is proper, it follows from the definitions that the subcategory $\mf{X}|_{O}$ coincides with $\mf{X}|_{O}^{\cov}$ (cf. \cref{defi_virtual_decuspidaloid_arises_from_object}(i)(ii)). Furthermore, every morphism in $\mf{X}|_{O}$ is of cov-type, and $\mf{X}|_{O}$ is a filtered category. By \cref{anabelian_CAVC_more}(i) (cf. also (b)), each object of $\mf{X}|_{O}$ is constructible. By \cref{anabelian_CAVC_more}(vi), for each morphism $f:P'\to P$ of $\mf{X}|_{O}$, one may construct the morphism between the corresponding canonical GDs induced by $f_{\ast}$. Therefore, the canonical schemes associated to $\Delta(O')\subset\Pi(O')$, where $O'$ ranges over the objects of $\mf{X}|_{O}$, together with the morphisms between the corresponding canonical GDs induced by $f_{\ast}$ for morphisms $f$ of $\mf{X}|_{O}$, define an inverse system of schemes (cf. \cref{anabelian_CAVC_more}(iv)).

    (d) By taking the direct limit of the rings of regular functions and fraction fields of the schemes appearing in the inverse system defined in step (c), one obtains an inclusion of fields
    \begin{displaymath}
        \base(\Pi)\injto \uc(\Pi),
    \end{displaymath}
    which admits a factorization 
    \begin{displaymath}
        \base(\Pi)\injto K(\sche(\Pi))\injto \uc(\Pi),
    \end{displaymath}
    where we write $\sche(\Pi) \coloneqq \sch(\Delta(O)\subset \Pi(O))$.
    Furthermore, $\base(\Pi)$ is the ring of regular functions of $\sche(\Pi)$. Hence, there exists a natural inclusion $k\injto \base(\Pi)$, which allows $\base(\Pi)$ to be regarded as an algebraic closure of $k$. Consequently, the scheme 
    \begin{displaymath}
        \sche(\Pi)
    \end{displaymath}
    is naturally isomorphic to 
    \begin{displaymath}
        (\sch(\Delta(O)\subset \Pi(O)))^{\Pi(O)/\Delta(O)} \times_{k} \base(\Pi).
    \end{displaymath}
    Denote by 
    \begin{displaymath}
        f: \sche(\Pi)\to \sp(\base(\Pi))
    \end{displaymath}
    the natural morphism induced by the natural field inclusion $\base(\Pi)\injto K(\sche(\Pi))$; by 
    \begin{displaymath}
        h: K(\sche(\Pi))\injto \uc(\Pi)
    \end{displaymath}
    the natural inclusion of fields.

    Moreover, the following assertions hold:

    (i) The action of $\Pi$ on $\base(\Pi)$ functorially induced by the conjugation action of $\Pi$ on itself factors through the quotient $\Pi\surjto G_{k}$. Moreover, there exists a unique $G_{k}$-equivariant isomorphism $\base(\Pi)\simto \bar{k}$ that is compatible with the natural inclusion $k\injto\base(\Pi)$ (cf. [\cite{MZK3}, Theorem 1.7(ii)(iii)]).
    
    (ii) The action of $\Pi$ on $\sche(\Pi)$ functorially induced by the conjugation action of $\Pi$ on itself factors through the quotient $\Pi\surjto G_{k}$. Moreover, there exists a natural $G_{k}$-equivariant isomorphism $\sche(\Pi)\simto X^{\cl}\times_{k}\bar{k}$ that is compatible with the isomorphism of (i).

    (iii) The action of $\Pi$ on $\uc(\Pi)$ functorially induced by the conjugation action of $\Pi$ on itself factors through the quotient $\Pi\surjto \Pi(O)$. Suppose that $\Pi(O)$ arises as the Galois group over $K(X^{\cl})$ of some FFUC of $X^{\cl}$ denoted by $\uc$ (cf. \cref{Notations and Terminologies}, the discussion entitled ``Fundamental Groups''). Then, there exists a natural $\Pi$-equivariant isomorphism $\uc(\Pi)\simto \uc$.
\end{theo}
\begin{proof}
    The algorithm and assertions stated in \cref{reconstruction_algorithm_genus_2} follow immediately from the various results cited in the statement of this algorithm.
    
\end{proof}

Next, we consider the case where the genus is $1$. First, we treat the situation in which the action of $\Pi_{[\pr^{2/1}_{X},s]}$ on $\cusp([\pr^{2/1}_{X},s])$ is trivial. In this case, one shows that every decuspidalization of $[\pr^{2/1}_{X},s]$ of type $(1,1)$ is a split pointed virtual curve (cf. \cref{defi_trivial_family}(ii)), admitting a splitting morphism onto the once-punctured curve associated to the unique smooth compactification $X^{\cl}$ of $X$, which is an open subscheme of an abelian variety. Moreover, for any two such decuspidalizations, one can characterize, in group-theoretic terms, a unique isomorphism between their respective virtual fundamental groups, which essentially corresponds to a ``translation'' of the abelian variety $X^{\cl}$ (cf. \cref{Notations and Terminologies}, the discussion entitled ``Fields, Schemes and Curves''). Then, by taking a suitable inverse limit and passing to the smooth compactification, one obtains an abelian variety that is naturally isomorphic to $X^{\cl}$. 


\begin{theo}\label{reconstruction_algorithm_genus_1}
    Let $k$ be an AF; $\bar{k}$ an algebraic closure of $k$; $G_{k}$ the Galois group $\gal(\bar{k}/k)$; $X$ a hyperbolic curve over $k$ of type $(1,r)$, where $r\geq 1$; $s:G_{k}\rightarrow \Pi_{X}$ a section of the natural surjection $\Pi_{X}\surjto G_{k}$. Write $\Pi$ for the virtual fundamental group $\Pi_{[\pr^{2/1}_{X},s]}$; $\Delta$ for the geometric virtual fundamental group $\Delta_{[\pr^{2/1}_{X},s]}$; $X^{\cl}$ for the unique smooth compactification of $X$ over $k$; $(\mf{X},F)\coloneqq(\cavc_{\Delta\subset \Pi}, F_{\Delta\subset \Pi})$ for the virtual decuspidaloid associated to $\Delta\subset \Pi$ (cf. \cref{CAVC_functor_defi}). Assume further that the action of $\Pi$ on $\cusp([\pr^{2/1}_{X},s])$ is trivial. Then one may construct various geometric objects (i.e., $\base(\Pi)$, $\sche(\Pi)$, and $f$ as defined in the subsequent algorithms) from the abstract profinite group $\Pi$ functorially with respect to isomorphisms of topological groups by the following algorithm:

    (a) One constructs $(\mf{X},F)$ from the abstract group $\Pi$ functorially with respect to isomorphisms of topological groups (cf. [\cite{SZM}, Theorem 5.1]; [\cite{SZM}, Theorem 5.13]; \cref{construction_CAVC}(ii)).

    (b) Denote by $\mf{S}$ the set of all subsets in $\cusp(\root(\mf{X}))$ (cf. \cref{prop_intrinsic_virtual_decuspidaloid}(ii)) whose complement in $\cusp(\root(\mf{X}))$ has cardinality $1$. Then every element of $\mf{S}$ is of admissible-type (cf. \cref{defi_admissible_type}). 
    Denote by $A\subset \obj(\mf{X})$ the set of decuspidalizations of $\root(\mf{X})$ induced by elements of $\mf{S}$ (cf. \cref{prop_intrinsic_virtual_decuspidaloid}(v)). It follows from \cref{1_1_is_constructible}(i) that every element of $A$ is constructible (cf. \cref{defi_constructible_object}). Let $R$, $R'$ be two distinct elements of $A$.
    By \cref{1_1_is_constructible}(ii), there exists a group isomorphism $\Pi(R)\simto \Pi(R')$ such that the (group-theoretic) conditions stated in \cref{1_1_is_constructible}(ii) hold. By \cref{anabelian_CAVC}(i) and \cref{generic_symmetry_genus_1}(i)(ii), 
    this group isomorphism is, in fact, unique up to an inner automorphism of $\Pi(R')$. Since $G_{k}$ is slim (cf. [\cite{MZK3}, Theorem 1.7(ii)(iii)]), it follows that this isomorphism is unique up to an inner automorphism of $\Delta(R')$. 
    Denote this unique $\Delta(R')$-outer morphism by $\sigma^{R,R'}$. By \cref{anabelian_CAVC}(ii), one constructs the isomorphism $[-1]^{R,R'}: \sch(\Delta(R)\subset\Pi(R))\simto \sch(\Delta(R)\subset\Pi(R)))$ associated to $\sigma^{R,R'}$. 
    Denote by $[-1]|_{R'}$ the automorphism of $\sch(\Delta(R')\subset\Pi(R'))$ (a hyperbolic curve of type $(1,1)$ whose unique smooth compactification is isomorphic to $X^{\cl}\times_{k}\bar{k}$ by \cref{1_1_is_constructible}(iii)) induced by the map given by multiplication by $-1$, where we regard $(\sch(\Delta(R')\subset\Pi(R')))^{\cl}\simto X^{\cl}\times_{k}\bar{k}$ as an abelian variety whose identity element is the unique rational cusp of $\sch(\Delta(R')\subset\Pi(R'))$. 
    Denote by $[1]^{R,R'}$ the composite $[-1]|_{R'}\circ [-1]^{R,R'}$. Let $R''$ be an element of $A$ distinct from both $R$ and $R'$. By \cref{generic_symmetry_genus_1}(ii)(iv), $[1]^{R,R''} = [1]^{R,R'}\circ [1]^{R',R''}$. 
    Therefore, we can construct a connected groupoid of schemes in which all automorphisms are trivial as follows:  the objects are the schemes $\sch(\Delta(R)\subset\Pi(R))$, as $R$ ranges over the elements of $A$, and the unique morphism from $\sch(\Delta(R)\subset\Pi(R))$ to $\sch(\Delta(R')\subset\Pi(R'))$ is given by $[1]^{R,R'}$. 
    Denote by
    \begin{displaymath}
        \sche(\Pi)
    \end{displaymath}
    the inverse limit of this connected groupoid (which is, in particular, a filtered category); by
    \begin{displaymath}
        \base(\Pi)
    \end{displaymath}
    the base field of $\sche(\Pi)$; by
    \begin{displaymath}
        f:\sche(\Pi) \to \sp(\base(\Pi))
    \end{displaymath}
    the naturally induced morphism.

    Moreover, the following assertions hold:

    (i) The action of $\Pi$ on $\base(\Pi)$ functorially induced by the conjugation action of $\Pi$ on itself factors through the quotient $\Pi\surjto G_{k}$. Moreover, there exists a $G_{k}$-equivariant isomorphism $\base(\Pi)\simto \bar{k}$.
    
    (ii) The action of $\Pi$ on $\sche(\Pi)$ functorially induced by the conjugation action of $\Pi$ on itself factors through the quotient $\Pi\surjto G_{k}$. Denote by $(\sche(\Pi))^{\cl}$ the unique smooth compactification of $\sche(\Pi)$ over $\base(\Pi)$. Then there exists a $G_{k}$-equivariant isomorphism $(\sche(\Pi))^{\cl}\simto X^{\cl}\times_{k}\bar{k}$.
\end{theo}
\begin{proof}
    The algorithm and assertions asserted in \cref{reconstruction_algorithm_genus_1} follow immediately from the various results cited in the statement of this algorithm.
\end{proof}

\begin{lemm}\label{1_1_is_constructible}
    In the situation of \cref{reconstruction_algorithm_genus_1}, let $O, O'\subset A$ be objects of $\mf{X}$ that lie in $A$. Then the following assertions hold:

    (i) $O$, $O'$ are constructible.

    (ii) Let $P$ (respectively, $P'$) be a field base change (cf. \cref{defi_intrinsic_virtual_decuspidaloid}(iii)) of $O$ (respectively $O'$) . Assume further that $G(P)$ coincides with $G(P')$, regarded as open subgroups of $G(\root(\mf{X}))$. Then, $(\Delta(P))^{\ab}$ coincides with $(\Delta(P'))^{\ab}$, regarded as subquotients of $\Pi(\root(\mf{X}))$.
    Moreover, there exists a group isomorphism $\sigma: \Pi(P)\simto \Pi(P')$ such that the following conditions hold:
    \begin{enumerate}
        \item The isomorphism $\sigma$ induces an isomorphism $\sigma^{\Delta}: \Delta(P)\simto \Delta(P')$.

        \item The isomorphism $\sigma$ induces the identity $G(P)=G(P')$ (\emph{not} merely up to an inner automorphism).

        \item The isomorphism $\sigma$ induces the endomorphism of $(\Delta(P))^{\ab}=(\Delta(P'))^{\ab}$ given by multiplication by $-1$.
    \end{enumerate}

    (iii) The unique compactification of the canonical scheme $(\sch(\Delta(O)\subset \Pi(O)))^{\Pi(O)/\Delta(O)}$ (cf. \cref{anabelian_CAVC}(ii)) over $\base(\Delta(O)\subset\Pi(O))^{\Pi(O)/\Delta(O)}$ is isomorphic to $X^{\cl}$. Moreover, $(\sch(\Delta(O)\subset \Pi(O)))^{\Pi(O)/\Delta(O)}$ is a hyperbolic curve of type $(1,1)$.
\end{lemm}
\begin{proof}
    First, we consider assertion (i). It suffices to show that $O$ is constructible. By definition, $\cusp(O)$ may be regarded as a subset of $\cusp(\root(\mf{X}))$ of cardinality $1$. Denote by $\alpha\in \cusp(\root(\mf{X}))$ the unique element of $\cusp(O)$; $\diag_{X}\in \Cusp(\root(\mf{X}))$ the element associated to the decuspidalization $[\pr^{2/1}_{X},s]\to [\pr^{1}_{X},s]$; $Q$ the decuspidalization of $\root(\mf{X})$ induced by $\{\diag_{X}\}$. By \cref{constructible_trivial_family}, since $[\pr^{1}_{X},s]$ is a split family of curves, $Q$ is constructible.
    If $\alpha\neq \diag_{X}$, then $O$ is a (possibly trivial) decuspidalization of $Q$. By \cref{constructible_via_subjugate}, $O$ is constructible. 
    Therefore, it suffices to consider the case where $\alpha = \diag_{X}$. Let $\beta\neq \alpha$ be an arbitrary element of $\cusp(\root(\mf{X}))$. 
    By \cref{constructible_trivial_family}, one may regard $\beta$ as an element of $\cusp(X)$. Since the action of $\Pi$ on $\cusp([\pr^{2/1}_{X},s])$ is assumed to be trivial, $\beta$ corresponds to a rational point $x$ of $X^{\cl}$. Denote by $X'$ the complement $X^{\cl}\setminus \{x\}$ of $x$ in $X^{\cl}$ (so $X'$ is of type $(1,1)$). Write $s':G_{k}\injto \Pi_{X'}$ for the section induced by $s$.
    Denote by $Q'$ the decuspidalization of $\root(\mf{X})$ induced by $\cusp(\root(\mf{X}))\setminus\{\alpha,\beta\}$. By repeated application of the isomorphism between the second lines of the first and second displays of [\cite{SZM}, Lemma 6.10(iii)], there exists an isomorphism $\phi:\Pi_{[\pr^{2/1}_{X'},s']} \simto \Pi(Q')$ such that the pair $([\pr^{2/1}_{X'},s'], \phi)$ is a VD of $\Delta(Q')\subset \Pi(Q')$. Denote by $O'$ the decuspidalization of $Q'$ induced by $\{\alpha\}$.
    By \cref{generic_symmetry_genus_1}(v)(vi) (where we think of $\beta$ as the identity element), there exists an isomorphism from $\Delta(O)\subset \Pi(O)$ to $\Delta(O')\subset \Pi(O')$ induced by the element ``$\tau$'' appearing in \cref{generic_symmetry_genus_1}(v)(vi). Moreover, by the previous arguments, $O'$ is constructible. Thus, $O$ is constructible.  This completes the proof of assertion (i).

    Next, we consider assertion (ii). We maintain the notations introduced in the proof of assertion (i). By passing to the base change corresponding to the field extension of $k$ determined by the equality $G(P)=G(P')$, one may assume without loss of generality that $P=O$ and $P'=O'$.  Also, we may assume without loss of generality that $O\neq O'$. If either $\cusp(O)$ or $\cusp(O')$ contains $\diag_{X}$, then the arguments involving ``$Q'$'' and ``$\tau$'' in the proof of assertion (i) already yield an example of $\sigma$ (cf. \cref{generic_symmetry_genus_1}(v)(vi)). Therefore, it suffices to consider the case where neither $\cusp(O)$ nor $\cusp(O')$ contains $\diag_{X}$.
    Then $O$ and $O'$ are two distinct decuspidalizations of $Q$ of type $(1,1)$. Since $\pr^{2}:X\times_{k}X\to X$ is a splitting morphism of $[\pr^{1}_{X},s]$, it follows from \cref{constructible_trivial_family} that it suffices to establish the following fact:

    ($*$): Let $Y$ be a smooth curve over $k$ of type $(1,2)$ with rational cusps. Let $x,y$ be the two points in its compactification $Y^{\cl}$. View $Y^{\cl}$ as an abelian variety by taking $x$ as the identity element. Define $\tau:Y\to Y$ by $z\mapsto y-z$. Then, up to an inner automorphism of $\Pi_{Y^{\cl}}$, the induced map $\tau_{\ast}$ coincides with the map given by multiplication by $-1$ on $\Delta_{Y^{\cl}}$.

    On the other hand, this assertion follows immediately from \cref{generic_symmetry_genus_1}(iv).  This completes the proof of assertion (ii).

    Next, we consider assertion (iii). We maintain the notations introduced in the proof of assertion (i). By assertion (ii), it suffices to show assertion (iii) for a single choice of $O$. Thus, we may assume that $O$ is a decuspidalization of $Q$. It then follows from \cref{constructible_trivial_family} that assertion (iii) holds.
\end{proof}


Finally, we obtain the following partial generalization of \cref{reconstruction_algorithm_genus_1} in the case where we no longer require that the action of $\Pi$ on $\cusp([\pr^{2/1}_{X},s])$ is trivial, but we do assume that $X$ admits a rational point.

\begin{theo}\label{validity_of_descent_genus_1}
    Let $k$ be an AF; $\bar{k}$ an algebraic closure of $k$; $G_{k}$ the Galois group $\gal(\bar{k}/k)$; $X$ a hyperbolic curve over $k$ of type $(1,r)$, where $r\geq 1$; $s:G_{k}\rightarrow \Pi_{X}$ a section of the natural surjection $\Pi_{X}\surjto G_{k}$. Write $\Pi$ for the virtual fundamental group $\Pi_{[\pr^{2/1}_{X},s]}$; $\Delta$ for the geometric virtual fundamental group $\Delta_{[\pr^{2/1}_{X},s]}$; $X^{\cl}$ for the unique smooth compactification of $X$ over $k$; $(\mf{X},F)\coloneqq(\cavc_{\Delta\subset \Pi}, F_{\Delta\subset \Pi})$ for the virtual decuspidaloid associated to $\Delta\subset \Pi$ (cf. \cref{CAVC_functor_defi}); $\Pi'\subset \Pi$ for the kernel of the natural action of $\Pi$ on $\cusp([\pr^{2/1}_{X},s])$; $k'$ for the finite Galois extension of $k$ corresponding to $\Pi'$ (where we note that $\Pi'$ contains $\Delta$); $Y$ for the field base change $X\times_{k}k'$;
    $s':G_{k'}\injto \Pi_{Y}$ for the restriction of $s$. 
    By construction, one verifies immediately that $\Pi'$ is naturally isomorphic to $\Pi_{[\pr^{2/1}_{Y},s']}$. Moreover, $\Pi_{[\pr^{2/1}_{Y},s']}$ satisfies the conditions of \cref{reconstruction_algorithm_genus_1}. Thus, one can construct $\base(\Pi')$, $\sche(\Pi')$, and $f:\sche(\Pi')\to\sp(\base(\Pi'))$ from the abstract profinite group $\Pi'$ in a manner that is functorial with respect to isomorphisms of topological groups. In particular, $\base(\Pi')$ and $\sche(\Pi')$ admit  natural $\Pi$-actions functorially induced by the conjugation action of $\Pi$ on $\Pi'$.
    Then the following assertions hold:

    (i) The action of $\Pi$ on $\base(\Pi')$ and $\sche(\Pi')$ factors through the quotient $\Pi\surjto G_{k}$. Moreover, there exists a $G_{k}$-equivariant isomorphism $\base(\Pi')\simto \bar{k}$.
    
    (ii) 
    Denote by $(\sche(\Pi'))^{\cl}$ the unique smooth compactification of $\sche(\Pi')$ over $\base(\Pi')$. Assume further that $X$ admits a rational point.  Then there exists a $G_{k}$-equivariant isomorphism $(\sche(\Pi'))^{\cl}\simto X^{\cl}\times_{k}\bar{k}$.
\end{theo}
\begin{proof}
    First, we consider assertion (i). Denote by $\diag_{X}\in \Cusp(\root(\mf{X}))$ the element associated to the decuspidalization $[\pr^{2/1}_{X},s]\to [\pr^{1}_{X},s]$. By definition, $\diag_{X}$ is stabilized by the action of $\Pi$. Therefore, $\{\diag_{X}\}$ is an admissible subset of $\Cusp(\root(\mf{X}))$. Denote by $O$ the decuspidalization of $\root(\mf{X})$ induced by $\Cusp(\root(\mf{X}))\setminus\{\diag_{X}\}$. Denote by $O'$ the unique field base change of $O$ such that the natural image of $G(O')$ in $G(\root(\mf{X})) = G_{k}$ is the open subgroup $G_{k'}$ induced by $k'$. By the content of the algorithm of \cref{reconstruction_algorithm_genus_1}, one verifies immediately that $O'$ is constructible, and that there exists a $G_{k'}$-equivariant isomorphism $\sch(\Delta(O')\subset\Pi(O'))\simto \sche(\Pi')$. 
    By \cref{anabelian_CAVC_more}(iii), $O$ is constructible.  Thus, we may apply \cref{anabelian_CAVC}.  By properties (1), (2) of \cref{anabelian_CAVC}(ii), the action of $\Pi$ on $\base(\Delta(O)\subset\Pi(O))$ and $\sche(\Delta(O)\subset\Pi(O))$ --- which, by the content of the algorithm of \cref{reconstruction_algorithm_genus_1}, may be identified, respectively and in a fashion compatible with the various actions of $\Pi$, with $\base(\Pi')$ and $\sche(\Pi')$ --- factors through the quotient $\Pi\surjto G_{k}$.  Finally, it follows from \cref{reconstruction_algorithm_genus_1}(i), together with \cref{Galois_group_geometricity_from_open_subgroup}(i), that $\base(\Delta(O)\subset\Pi(O))$ admits a $G_{k}$-equivariant isomorphism with $\bar{k}$.  This completes the proof of assertion (i).


    Next, we consider assertion (ii). 
    Let $S\subset\Cusp(\root(\mf{X}))$ be a nonempty subset stabilized by the action of $\Pi$ such that $\diag_{X}\notin S$. Denote by $P$ the decuspidalization of $\root(\mf{X})$ induced by $\Cusp(\root(\mf{X}))\setminus S$ (cf.\ \cref{prop_intrinsic_virtual_decuspidaloid}(v)); by $P'$ the unique field base change (cf.\ \cref{defi_intrinsic_virtual_decuspidaloid}(iii)) of $P$ such that the natural image of $G(P')$ in $G(\root(\mf{X})) = G_{k}$ is the open subgroup $G_{k'}$ induced by $k'$.
    By \cref{anabelian_CAVC_more}(ii) and [\cite{SZM}, Proposition 6.1], $P$ is constructible. Moreover, since the naturally induced morphism $\Pi(P')\to\Pi(P)$ is of geo-iso-type (cf.\ \cref{defi_intrinsic_virtual_decuspidaloid}(iii); \cref{defi_CAVC_morphism_type}(i)), hence of cov-type (cf.\ \cref{abs_nonsense_CAVC_morphism}(ii)), it follows from \cref{anabelian_CAVC_more}(i) that $P'$ is constructible. Note that one may regard $\cusp(P')$ as the set $S$.

    Denote by $W$ (respectively, $W'$) the unique smooth compactification of the canonical scheme $(\sch(\Delta(P)\subset\Pi(P)))^{G(P)}$ (respectively, $(\sch(\Delta(P')\subset\Pi(P')))^{G(P')}$) (cf.\ \cref{anabelian_CAVC}(ii)). By \cref{anabelian_CAVC_more}(iv)(vi), $W'$ is naturally isomorphic to $W\times_{k}k'$. Moreover, by \cref{anabelian_CAVC_more}(v)(vii) and [\cite{SZM}, Proposition 6.1], $W$ is naturally isomorphic to $X^{\cl}$.

    Denote by $V$ the unique smooth compactification of $(\sche(\Pi'))^{G_{k'}}$ over $(\base(\Pi'))^{G_{k'}}$ (cf.\ \cref{reconstruction_algorithm_genus_1}(ii)); by $a$ the $k'$-rational point of $V$ corresponding to the unique cusp of $(\sche(\Pi'))^{G_{k'}}$.
    Let $\alpha\in S$ be an element. Then $\alpha$ corresponds to a $k'$-rational point $p_{\alpha}\in W'$ (where we note that $\Pi'$ acts trivially on $\cusp(P')=S$, by the definition of $k'$). Denote by $Q'_{\alpha}$ the decuspidalization of $P'$ induced by $S\setminus\{\alpha\}$ (cf.\ \cref{prop_intrinsic_virtual_decuspidaloid}(v)). Then $Q'_{\alpha}$ is of type $(1,1)$, and one verifies immediately that $Q'_{\alpha}$ belongs to the set ``$A$'' defined in \cref{reconstruction_algorithm_genus_1}(b).

    Denote by $F_{\alpha}$ the isomorphism $W'\simto V$ obtained by taking the isomorphism induced between the respective unique smooth compactifications by the natural isomorphism
    \begin{displaymath}
        (\sch(\Delta(Q'_{\alpha})\subset\Pi(Q'_{\alpha})))^{G(Q'_{\alpha})}\simto(\sche(\Pi'))^{G_{k'}},
    \end{displaymath}
    where we note that $\sch(\Delta(Q'_{\alpha})\subset\Pi(Q'_{\alpha}))$ is an object of the filtered category whose inverse limit is the scheme $\sche(\Pi')$ (cf.\ \cref{reconstruction_algorithm_genus_1}(b)), and that the smooth compactification of $(\sch(\Delta(Q'_{\alpha})\subset\Pi(Q'_{\alpha})))^{G(Q'_{\alpha})}$ is naturally isomorphic to $W'$ (cf.\ \cref{anabelian_CAVC_more}(v)(vii), applied to the morphism of decusp-type from $P'$ to $Q'_{\alpha}$). By construction, $F_{\alpha}(p_{\alpha})=a$.

    Now, let $\alpha,\beta\in S$ be two (not necessarily distinct) elements. Denote by $\sigma^{\alpha,\beta}$ the automorphism of $W'$ such that $F_{\beta}\circ \sigma^{\alpha,\beta} = F_{\alpha}$. If $\alpha=\beta$, then $\sigma^{\alpha,\alpha}=\id$. If $\alpha\neq\beta$, then, by chasing the group isomorphisms induced by the geometric fundamental groups of the various proper curves that appear in the discussion --- where we observe that all such geometric fundamental groups may be canonically identified with the same quotient of $\Delta$ (cf.\ \cref{generic_symmetry_genus_1}(i)) --- one verifies immediately from \cref{generic_symmetry_genus_1}(iii) that $\sigma^{\alpha,\beta}$ is a translation (cf.\ \cref{Notations and Terminologies}, the discussion entitled ``Fields, Schemes and Curves''). Thus, regardless of whether or not $\alpha=\beta$, $\sigma^{\alpha,\beta}$ is the unique translation sending $p_{\alpha}$ to $p_{\beta}$. In particular, $\sigma^{\alpha,\beta}\circ \sigma^{\beta,\alpha} = \id$.

    Finally, we consider the action of $\gal(k'/k)\coloneqq G_{k}/G_{k'}$ on $V$. Let $\theta\in \gal(k'/k)$. Denote by $\theta_{V}$ the image of $\theta$ under the natural action of $\Pi$ on $((\sche(\Pi'))^{G_{k'}})^{\cl}$ (cf.\ assertion (i)); similarly, denote by $\theta_{W'}$ the image of $\theta$ under the automorphism of $W'$ induced by the natural isomorphism $W'\simto W\times_{k}k'$. Suppose that $\theta$ sends $\alpha$ to $\beta$, i.e., $\theta(\alpha)=\beta$, under the natural action on $S$. By the functoriality of the construction of $\sche(\Pi')$ (cf.\ \cref{reconstruction_algorithm_genus_1}(b); \cref{anabelian_CAVC}(ii)) with respect to the conjugation action of $\Pi$ on $\Pi'$, we have
    \begin{displaymath}
        \theta_{V}\circ F_{\alpha} = F_{\beta}\circ \theta_{W'}.
    \end{displaymath}
    Combining this identity with $F_{\beta} = F_{\alpha}\circ \sigma^{\beta,\alpha}$ (which follows from $\sigma^{\alpha,\beta}\circ\sigma^{\beta,\alpha}=\id$), we obtain
    \begin{displaymath}
        \theta_{V}\circ F_{\alpha} = F_{\alpha}\circ \sigma^{\beta,\alpha}\circ \theta_{W'}.
    \end{displaymath}
    Recall from the above discussion that $W$ is naturally isomorphic to $X^{\cl}$.  Thus, by [\cite{SZM}, Corollary 6.7] and the equation of the above display, to prove assertion (ii), it suffices to verify that the two group homomorphisms
    \begin{displaymath}
        s^{\dagger}, s^{\ddagger}: \gal(k'/k)\longrightarrow \Aut_{k}(W'),
    \end{displaymath}
    where $\Aut_{k}(W')$ denotes the group of $k$-automorphisms of $W'$,  defined by
    \begin{displaymath}
        s^{\dagger}(\theta)\coloneqq\theta_{W'} \quad\text{and}\quad s^{\ddagger}(\theta)\coloneqq\sigma^{\theta(\alpha),\alpha}\circ\theta_{W'},
    \end{displaymath}
    where $\alpha\in S$ denotes a fixed element (cf.\ the definition of $F_{\alpha}$), are conjugate by an element of $\Aut_{k}(W')$. 
    (Here, we note that it follows from the equation $\theta_{V}\circ F_{\alpha} = F_{\alpha}\circ \sigma^{\theta(\alpha),\alpha}\circ \theta_{W'}$ that $s^{\ddagger}$ is indeed a group homomorphism.)
    On the other hand, by assumption, $X$ admits a $k$-rational point. Consequently, $W\simto X^{\cl}$ admits a $k$-rational point $x$. Denote by $x'$ the natural preimage of $x$ in $W'$. Then $\theta_{W'}(x')=x'$ for every $\theta\in\gal(k'/k)$ (since $x$ is a $k$-rational point of $W$), and $\theta_{W'}(p_{\alpha})=p_{\theta(\alpha)}$ (since the cusps of $\cusp(P')=S$ are permuted by $\theta$ compatibly with the corresponding $k'$-rational points of $W'$). Let $T$ denote the unique translation (cf.\ \cref{Notations and Terminologies}, the discussion entitled ``Fields, Schemes and Curves'') of $W'$ sending $x'$ to $p_{\alpha}$. Then, for every $\theta\in\gal(k'/k)$,
    \begin{displaymath}
        T\circ s^{\dagger}(\theta)\circ T^{-1} = T\circ\theta_{W'}\circ T^{-1} = \sigma^{\theta(\alpha),\alpha}\circ\theta_{W'} = s^{\ddagger}(\theta),
    \end{displaymath}
    where the second equality follows from the identity $\theta_{W'}\circ T_{a}\circ\theta_{W'}^{-1} = T_{\theta_{W'}(a)}$ (which holds since $\theta_{W'}$ preserves the group law on $W'$; here, we write $T_{a}$ for the translation by a point $a$) applied to $T=T_{p_{\alpha}-x'}$:
    \begin{displaymath}
        T\circ\theta_{W'}\circ T^{-1} = T_{(p_{\alpha}-x')-\theta_{W'}(p_{\alpha}-x')}\circ\theta_{W'} = T_{p_{\alpha}-p_{\theta(\alpha)}}\circ\theta_{W'} = \sigma^{\theta(\alpha),\alpha}\circ\theta_{W'}.
    \end{displaymath}
    This completes the proof of assertion (ii).

\end{proof}







\bibliographystyle{plain}
\bibliography{arithmetic_lib.bib}

\end{document}